\documentclass[10pt,a4paper]{amsart}
\usepackage{verbatim}

\usepackage[T1]{fontenc}
\usepackage{graphicx}
\usepackage{enumerate}
\usepackage{amsmath,amsfonts,amssymb}
\usepackage{color}
\usepackage{cite}
\definecolor{citation}{rgb}{0.2,0.58,0.2} 
\definecolor{formula}{rgb}{0.1,0.2,0.6}
\definecolor{url}{rgb}{0.3,0,0.5}
\usepackage{pgf,tikz}
\usepackage{mathrsfs}
\usepackage{fancyhdr}

\usepackage{dsfont}

\newcommand{\reqnomode}{\tagsleft@false}

\usepackage[colorlinks,pdfpagelabels,pdfstartview = FitH,bookmarksopen = true,bookmarksnumbered = true,urlcolor=url,linkcolor = formula,plainpages = false,hypertexnames = false,citecolor = citation] {hyperref}

\def\dist{\,{\rm dist}}

\allowdisplaybreaks
\makeatletter
\DeclareRobustCommand*{\bfseries}{%
  \not@math@alphabet\bfseries\mathbf
  \fontseries\bfdefault\selectfont
  \boldmath
}

\makeatother

\newlength{\defbaselineskip}
\newcommand{\setlinespacing}[1]
           {\setlength{\baselineskip}{#1 \defbaselineskip}}
\newcommand{\mint}{\mathop{\int\hskip -1,05em -\, \!\!\!}\nolimits}

\newtheorem{theorem}{Theorem}
\newtheorem{corollary}{Corollary}[section]
\newtheorem{definition}{Definition}
\newtheorem{remark}{Remark}[section]
\newtheorem{lemma}{Lemma}[section]
\newtheorem{proposition}{Proposition}[section]
\numberwithin{equation}{section}
\allowdisplaybreaks[4]

\newcommand\eps\varepsilon
\def\eqn#1$$#2$${\begin{equation}\label#1#2\end{equation}}

\newcommand{\be}{\begin{equation}}
\newcommand{\ee}{\end{equation}}

\newcommand{\snr}[1]{\lvert #1\rvert}

\newcommand{\nr}[1]{\lVert #1 \rVert}

\newcommand{\RN}{\mathbb{R}^{N}}

\newcommand{\N}{\mathbb{N}}

\def\name[#1, #2]{#1 #2}

\title[
$\omega$-minima of functionals with $\varphi$-growth]{On the regularity of the $\omega$-minima of $\varphi$-functionals
}
\author{Cristiana De Filippis}  \address{Cristiana De Filippis\\Mathematical Institute, University of Oxford\\ Andrew Wiles Building, Radcliffe Observatory Quarter, Woodstock Road, Oxford, OX26GG, Oxford, United Kingdom} \email{\texttt{Cristiana.DeFilippis@maths.ox.ac.uk}}

\begin{document}

\subjclass[2010]{35J60, 35J70\vspace{1mm}} 

\keywords{Regularity, $\omega$-minima, $\varphi$-functionals, Orlicz spaces\vspace{1mm}}

\thanks{{\it Acknowledgements.}\ The author is supported by the Engineering and Physical Sciences Research Council (EPSRC): CDT Grant Ref. EP/L015811/1. 
\vspace{1mm}}

\maketitle

\begin{abstract}
 We focus on some regularity properties of $\omega$-minima of variational integrals with $\varphi$-growth and provide an upper bound on the Hausdorff dimension of their singular set.
\end{abstract}
\vspace{3mm}

\setlinespacing{1.08}

\section{Introduction}\label{intro}
In this paper we study some regularity properties of $\omega$-minima of certain functionals with $\varphi$-growth, i.e., variational integrals whose integrand is modelled on an $N$-function $\varphi$; see Section \ref{assec} for an overview of its main properties. Precisely, we shall focus on two classes of non autonomous functionals:
\begin{flalign}\label{fullf}
W^{1,\varphi}(\Omega,\RN)\ni w &\mapsto   \mathcal{F}(w,\Omega):= \int_{\Omega}f(x,w,Dw) \ dx,\\
W^{1,\varphi}(\Omega,\RN)\ni w &\mapsto \mathcal{G}(w,\Omega):= \int_{\Omega}g(x,Dw) \ dx,\label{xf}
\end{flalign}
where $\Omega\in \mathbb{R}^{n}$ is an open set, $n\ge 2$, $N\ge 1$, $f\colon \Omega\times \RN\times \mathbb{R}^{N\times n}\to \mathbb{R}$ and $g\colon \Omega\times \mathbb{R}^{N\times n}\to \mathbb{R}$ are continuous integrands, see again Section \ref{assec} for the precise set of hypotheses and relevant definitions. Under polynomial growth assumptions, the regularity theory for minimizers of $\varphi$-functionals falls in the realm of the theory of variational integrals with non-standard growth, which was started by Marcellini's seminal works \cite{M3,M4,M5} and it is by now very rich. See \cite{Baronicv, BCM2,CD,CKP1,CKP2,CM1,DO,DP,DVS} for an (incomplete) account of the most recent advances in this field, \cite{D1,DM,DSV} for the case of manifold-constrained problems and critical systems and to \cite{Dark} for a reasonable survey concerning the regularity of minima under standard and non-standard growth  conditions. Needless to say, the most treated model example is given by $\varphi(t)=t^p$, where the functional in question 
$$
W^{1,\varphi}(\Omega,\RN)\ni w \mapsto  \int_{\Omega}|Dw|^p \ dx\;, 
$$
has been studied at length over the years and whose associated Euler-Lagrange equation defines the well-known $p$-Laplacean operator. We refer to \cite{KuMi1,KuMi2} for a rather comprehensive account of recent regularity theory. For the ease of notation, we denote \eqref{fullf} or \eqref{xf} by
\begin{flalign}\label{h}
\mathcal{K}(w,\Omega):=\int_{\Omega}k(x,w,Dw) \ dx, \ \ k(\cdot)\in \{f(\cdot),g(\cdot)\},
\end{flalign}
and of course, when $k(\cdot)\equiv g(\cdot)$ no dependency from the second variable occurs. The classical notion of minimizer for $u \in W^{1,\varphi}(\Omega,\RN)$ then simply requires that 
\begin{flalign*}
\int_{\Omega_{0}}k(x,u,Du) \ dx \le \int_{\Omega_{0}}k(x,w,Dw) \ dx
\end{flalign*}
holds for all $w \in u+W^{1,\varphi}_{0}(\Omega_{0},\RN)$ and all $\Omega_{0}\Subset \Omega$, while the concept of $\omega$-minimizer is rather more delicate and prescribes that
\begin{flalign*}
\int_{B_{r}}k(x,u,Du) \ dx \le (1+\omega(r))\int_{B_{r}}k(x,w,Dw) \ dx,
\end{flalign*}
for all $w\in u+W^{1,\varphi}_{0}(B_{r},\RN)$ and all balls $B_{r}\Subset \Omega$. Here $\omega \colon [0,\infty)\to[0,\infty)$ is a non decreasing, concave function such that $\lim_{r\to 0}\omega(r)=0$. See \eqref{om} for more details. The notion of $\omega$-minimizers is an extension of the concept of minimizer: it was introduced in the framework of Geometric Measure Theory, \cite{A, B} and then studied in the non-parametric setting by several authors, see \cite[Chapters 7, 8, 9]{G} for an introduction and \cite{Dark} for a list of references. The interest raised on the question of regularity for $\omega$-minimizers is motivated by the fact that, in certain situations, minimizers of constrained variational problems can be realized as $\omega$-minimizers of unconstrained problems, thus significantly simplifying the treatment. This is the case, for instance, of obstacle problems and volume constrained minimizers, as first noted in the setting of Geometric Measure Theory \cite{A,B} and then in the setting of variational integrals \cite{Anzellotti, DGG}. Eventually, an increasing number of papers has been dedicated to the study of regualrity properties of $\omega$-minima both in the scalar and in the vectorial case \cite{BKM, DEF, KM1, KM2, OK2, OK3}. Here we investigate some regularity properties of the $\omega$-minima of functionals \eqref{fullf}-\eqref{xf}. Precisely, when considering variational integrands like \eqref{xf}, we derive fractional differentiability for a certain function of $Du$. This is shown in Theorem \ref{frac}.
\begin{theorem}\label{frac}
Let $u\in W^{1,\varphi}(\Omega,\RN)$ be an $\omega$-minimizer of \eqref{xf}, under assumptions \eqref{assg} and \eqref{om}. Then 
\begin{flalign}\label{res}
V_{\varphi}(Du)\in W^{\delta,2}_{\mathrm{loc}}(\Omega,\mathbb{R}^{N\times n})\cap W^{1,\frac{2n}{n-2\delta}}_{\mathrm{loc}}(\Omega,\mathbb{R}^{N\times n})
\end{flalign}
for all $\delta \in (0,\sigma)$, where $\sigma:=\min\left\{\alpha,\frac{\gamma}{2+\gamma},\frac{\gamma(1+\alpha)}{2(1+\gamma)}\right\}$.
\end{theorem}
The content of the previous theorem quantifies the interaction between the H\"older continuity exponent of the map $x\mapsto g(x,\cdot)$ and the rate of decay of $\omega(\cdot)$ at zero, resulting in the fractional differentiability of $V_{\varphi}(Du)$. Then, via Sobolev embedding, we also prove higher integrability for $V_{\varphi}(Du)$. We do not know whether this result is optimal: however it might be the correct one. In fact, if $\alpha$ is the H\"older exponent of the map $x\mapsto g(x,\cdot)$, cf. $\eqref{assg}_{2}$ and $\gamma$ controls the decay at zero of $\omega(\cdot)$, see \eqref{om}, then when $\alpha=1$, we get $\sigma=\frac{\gamma}{2+\gamma}$, obtained in \cite[Lemma 3.1]{KM1}, while, as $\gamma\to \infty$, we end up with $\sigma=\alpha$, which can be retrieved in \cite{CM1, ELM}. Extra fractional differentiability is not only interesting \emph{per se} as a result, but it can be crucial in order to show finer regularity properties, see \cite{ABES,AKM,BL,BLS,DKM,KuMi3,Min2}. On the other hand, if we consider quasilinear structures as the one characterizing \eqref{fullf}, we can provide a partial regularity result, i.e.: $V_{\varphi}(Du)$ is locally H\"older continuous on an open subset $\Omega_{u}\subset \Omega$ of full $n$-dimensional Lebesgue's measure, together with an intrinsic description of its complementary $\Sigma_{u}:=\Omega\setminus \Omega_{u}$, the singular set. In fact, we have:
\begin{theorem}\label{reg}
Let $u \in W^{1,\varphi}(\Omega,\RN)$ be an $\omega$-minimizer of \eqref{fullf}, under assumptions \eqref{assf} and \eqref{omt}-\eqref{om}. Then there exists an open set $\Omega_{u}\subset \Omega$ of full $n$-dimensional Lebesgue measure such that 
\begin{flalign*}
&V_{\varphi}(Du) \in C^{0,\beta_{0}}_{\mathrm{loc}}(\Omega_{u},\mathbb{R}^{N\times n}),\quad \beta_{0}\in (0,1),\\
&Du\in C^{0,\beta'}_{\mathrm{loc}}(\Omega_{u},\mathbb{R}^{N\times n}),\quad \beta'\in (0,1),
\end{flalign*}
with $\beta',\beta_{0}=\beta',\beta_{0}(n,N,\nu,L,\Delta_{2}(\varphi),\Delta_{2}(\varphi,\varphi^{*}),c_{\varphi},\alpha,\gamma)$. Precisely, the regular set $\Omega_{u}$ can be characterized as the set of points $x_{0}\in \Omega$ such that
\begin{flalign}\label{rs}
\varphi(\varepsilon/\varrho)^{-1}\mint_{B_{\varrho}(x_{0})}\varphi(\snr{Du}) \ dx<1,
\end{flalign}
for some small $\varepsilon=\varepsilon(n,N,\nu,L,\Delta_{2}(\varphi),\Delta_{2}(\varphi,\varphi^{*}),c_{\varphi})$.
\end{theorem}
From \eqref{rs}, it follows the inclusion:
\begin{flalign*}
\Sigma_{u}\subset \left\{x_{0}\in \Omega\colon \limsup_{\varrho\to 0}\frac{1}{\varphi(\varrho^{-1})}\mint_{B_{\varrho}(x_{0})}\varphi(\snr{Du}) \ dx>0\right\},
\end{flalign*}
which is fundamental as to determine an upper bound for the Hausdorff dimension of $\Sigma_{u}$. The paper is organized as follows: in Section \ref{assec} we describe our framework, fully detail the problem and list the main assumptions we adopt. In Section \ref{pre} we collect some preliminary results, well known to experts which will be crucial for the proof. Section \ref{basic} essentially contains the basic regularity results such as Caccioppoli's inequality and Gehring's lemma. Finally, Sections \ref{te1} and \ref{te2} are devoted to the proof of Theorem \ref{frac} and Theorem \ref{reg} respectively.

\section{Notation and main assumptions}\label{assec}
In what follows we denote by $c$ a general positive constant, possibly varying from line to line; special occurrences will be denoted by $c_{1}, c_{*}, \bar{c}$ or the like. All such constants will always be larger or equal than one; moreover, relevant dependencies on parameters will be emphasized using parentheses. By $B_{r}(x_{0}):=\left\{x \in \mathbb{R}^{n}\colon \snr{x-x_{0}}<r\right\}$ we mean the open ball with centre $x_{0}$ and radius $r>0$; when not relevant, or clear from the context, we shall omit denoting the centre, writing just $B_{r}(x_{0})\equiv B_{r}$. Moreover, $Q_{r}(x_{0}):=\left\{x \in \mathbb{R}^{n}\colon \snr{x^{i}-x_{0}^{i}}<r, \ i \in \{1,\cdots,n\}\right\}$ will indicate the open cube having side length $2r$, center $x_{0}$ and sides parallel to the axes. Unless otherwise stated, different balls or cubes in the same context will have the same centre. For $a_{1},a_{2}$ scalar or vectors and $\lambda \in [0,1]$, we refer to the segment joining $a_{1}$ and $a_{2}$ with the symbol $[a_{1},a_{2}]_{\lambda}:=\lambda a_{1}+(1-\lambda)a_{2}$. With $\tilde{\Omega}\subset \mathbb{R}^{n}$, $n\ge1$, being a set of positive, finite Lebesgue measure $\snr{\tilde{\Omega}}>0$ and with $w\colon \tilde{\Omega}\to \mathbb{R}^{N}$, $N\ge 1$ being a measurable map, we shall denote its integral average by
\begin{flalign*}
(w)_{\tilde{\Omega}}:=\mint_{\tilde{\Omega}}w(x) \ dx =\frac{1}{\snr{\tilde{\Omega}}}\int_{\tilde{\Omega}}w(x) \ dx.
\end{flalign*}
As usual, if $v\colon \Omega \to \mathbb{R}^{N}$ is any $\gamma$-H\"older continuous map with $\gamma \in (0,1]$ and $A \subset \Omega$, then its H\"older seminorm is defined as
\begin{flalign*}
[v]_{0,\gamma;A}:=\sup_{x,y\in A, \ x \not = y}\frac{\snr{v(x)-v(y)}}{\snr{x-y}^{\gamma}}.
\end{flalign*}
When not relevant or clear from the context, we shall avoid mentioning the dependence of the H\"older seminorm from $A\subset \Omega$ by simply writing: $[v]_{0,\gamma}:=[v]_{0,\gamma;A}$.
\vspace{1mm}
\subsection{Fractional Sobolev spaces} For $h \in \mathbb{R}^{n}\setminus \{0\}$, if $G\colon \Omega\to \RN$ is any vector field, we define the finite difference operators as
\begin{flalign*}
\tau_{h}G(x):=G(x+h)-G(x) \quad \mathrm{and}\quad \tau_{-h}G(x)=G(x)-G(x-h).
\end{flalign*}
This makes sense whenever $x$, $x+h$ and $x-h$ belong to $\Omega$, an assumption that will be satisfied whenever we use $\tau_{h}$. If $\Omega\subset \mathbb{R}^{n}$ is a bounded, open set, given $p>1$ and $\sigma \in (0,1)$ we say that $w \in W^{\sigma,p}(\Omega,\RN)$ provided that $w\in L^{p}(\Omega,\RN)$ and the Gagliardo norm of $w$:
\begin{flalign*}
\nr{w}_{W^{\sigma,p}(\Omega,\RN)}=\left(\int_{\Omega}\snr{w}^{p} \ dx\right)^{1/p}+\left(\int_{\Omega}\int_{\Omega}\frac{\snr{w(x)-w(y)}^{p}}{\snr{x-y}^{n+\sigma p}} \ dxdy\right)^{1/p}
\end{flalign*}
is finite. The local variant $W^{\sigma,p}_{\mathrm{loc}}(\Omega,\RN)$ is defined in the usual way. For more details on this matter see \cite{DPV}.
\vspace{1mm}
\subsection{$N$-functions}\label{N} We consider a convex function $\varphi\colon [0,\infty)\to [0,\infty)$, such that
\begin{flalign}\label{phi1}
\varphi \in C^{1}([0,\infty))\cap C^{2,\theta}((0,\infty)), \quad \varphi(0)=\varphi'(0)=0, \quad t\mapsto\varphi'(t) \ \mathrm{increasing}, \quad \lim_{t\to \infty}\varphi'(t)=\infty.
\end{flalign}
We say that $\varphi$ satisfies the $\Delta_{2}$-condition if there exists a constant $c>0$ such that for all $t\ge 0$ there holds $\varphi(2t)\le c\varphi(t)$. By $\Delta_{2}(\varphi)$ we denote the smallest of such constants. Since $\varphi(t)\le \varphi(2t)$, the $\Delta_{2}$-condition is equivalent to $\varphi(t)\sim \varphi(2t)$ for all $t\ge 0$, where the constants implicit in "$\sim$" depend only from the characteristics of $\varphi$. By $(\varphi')^{-1}\colon [0,\infty)\to [0,\infty)$ we mean the function $(\varphi')^{-1}(t):=\left\{s \in [0,\infty)\colon \varphi(s)\le t\right\}$. If $\varphi'$ is strictly increasing, then $(\varphi')^{-1}$ is the inverse function of $\varphi'$. Then $\varphi^{*}\colon [0,\infty)\to [0,\infty)$, defined as $\varphi^{*}(t):=\int_{0}^{t}(\varphi')^{-1}(s) \ ds$ is again an $N$-function and $(\varphi^{*})'(t)=(\varphi')^{-1}(t)$ for $t>0$. It is the complementary function of $\varphi$. Note that $\varphi^{*}(t)=\sup_{s\ge 0}\left\{ st-\varphi(s) \right\}$ and $(\varphi^{*})^{*}=\varphi$. For all $\delta>0$, there exists a $c_{\delta}$, only depending on $\delta$ and on $\Delta_{2}(\varphi,\varphi^{*})$, such that for all $t,s\ge 0$ there holds
\begin{flalign}\label{you}
ts\le \delta \varphi(t)+c_{\delta}\varphi^{*}(s), \quad ts\le \delta \varphi^{*}(t)+c_{\delta}\varphi(s).
\end{flalign}
For $\delta=1$, $c_{\delta}=1$. This is Young's inequality. For all $t\ge 0$,
\begin{flalign*}
\begin{cases}
\ (t/2)\varphi(t/2)\le \varphi(t)\le t\varphi'(t)\\
\ \varphi\left(\frac{\varphi^{*}(t)}{t}\right)\le \varphi^{*}(t)\le \varphi\left(\frac{2\varphi^{*}(t)}{t}\right)
\end{cases}.
\end{flalign*}
Therefore, uniformly in $t\ge 0$,
\begin{flalign}\label{af2}
\varphi(t)\sim t\varphi'(t)\quad \mathrm{and}\quad \varphi^{*}(\varphi'(t))\sim \varphi(t), 
\end{flalign}
with constants depending only on $\Delta_{2}(\varphi,\varphi^{*})$. If $\varphi$ and $\rho$ are $N$-functions with $\varphi(t)\le \rho(t)$ for all $t\ge 0$, then 
\begin{flalign*}
\rho^{*}(t)\le \varphi^{*}(t) \ \ \mathrm{for \ all \ }t\ge 0.
\end{flalign*}
We also assume that
\begin{flalign}\label{phi6}
\varphi'(t)\sim t\varphi''(t) \ \ \mathrm{for \ all \ }t\ge 0,
\end{flalign}
where the constants implicit in "$\sim$" depend only on $c_{\varphi}$, a constant describing the characteristics of $\varphi$, and that $\varphi''$ is H\"older-continuous away from zero, i.e.:
\begin{flalign}\label{aff5}
\snr{\varphi''(s+t)-\varphi''(t)}\le L\varphi''(t)\left(\frac{\snr{s}}{t}\right)^{\vartheta} \ \ \mathrm{for some} \ \ \vartheta \in (0,1],
\end{flalign}
for all $t>0$ and $\snr{s}<t/2$. Here, $L$ is a positive, absolute constant. Let us discuss the natural functional setting related to functionals \eqref{fullf}-\eqref{xf}.
\begin{definition}\label{d1}
If $\Omega \subset \mathbb{R}^{n}$ is any open set and $w\colon \Omega\to \RN$ is a measurable map, consider the Luxemburg norm
\begin{flalign*}
\nr{w}_{L^{\varphi}(\Omega)}:=\inf \left\{\lambda >0\colon \int_{\Omega}\varphi (\snr{w}/\lambda) \ dx \le 1\right\}.
\end{flalign*}
Then, the space $L^{\varphi}(\Omega,\RN)$ is defined as
\begin{flalign*}
L^{\varphi}(\Omega,\RN):=\left\{ \ w\colon \Omega\to \RN \ \mbox{measurable, such that} \ \nr{w}_{L^{\varphi}(\Omega)}<\infty  \ \right\},
\end{flalign*}
while the Orlicz-Sobolev space $W^{1,\varphi}(\Omega,\mathbb{R}^{N})$ is naturally defined as
\begin{flalign*}
W^{1,\varphi}(\Omega,\mathbb{R}^{N}):=\left\{w\in L^{\varphi}(\Omega,\RN) \ \mbox{such that} \ Dw \in L^{\varphi}(\Omega,\mathbb{R}^{N\times n})\right\}
\end{flalign*}
with norm
\begin{flalign*}
\nr{w}_{W^{1,\varphi}(\Omega,\RN)}:=\inf\left\{\lambda>0\colon \int_{\Omega}\varphi(\snr{w}/\lambda)+\varphi(\snr{Dw}/\lambda) \ dx\le 1 \right\},
\end{flalign*}
and $W^{1,\varphi}_{0}(\Omega,\RN)$ is the closure of $C^{\infty}_{c}(\Omega,\RN)$ with respect to the above norm. The local variant of those spaces is then defined in an obvious way.
\end{definition}
Notice that, under the assumption listed so far, it is easy to see that there are $1<s_{0}\le q_{0}$ so that $t^{s_{0}}\le c(1+\varphi(t))\le c(1+t^{q_{0}})$, for $c=c(\Delta_{2}(\varphi),\Delta_{2}(\varphi,\varphi^{*}))$, see \cite{DVS, HHT}, so the spaces in Definition \ref{d1} can be equivalently defined as
\begin{flalign*}
&L^{\varphi}(\Omega,\RN):=\left\{w\colon \Omega\to \RN \ \mbox{measurable}\colon \varphi(\snr{w})\in L^{1}(\Omega)\right\}, \\ 
&W^{1,\varphi}(\Omega,\mathbb{R}^{N}):=\left\{w\in L^{\varphi}(\Omega,\RN)\colon Dw\in L^{\varphi}(\Omega,\mathbb{R}^{N\times n})\right\}.
\end{flalign*}
We define alsto the auxiliary vector field $V_{\varphi}\colon \mathbb{R}^{N\times n}\to \mathbb{R}^{N\times n}$ by
\begin{flalign}\label{vphi}
V_{\varphi}(z):=\left(\frac{\varphi'(\snr{z})}{\snr{z}}\right)^{\frac{1}{2}}z \ \ \mbox{for \ any \ }z \in \mathbb{R}^{N\times n};
\end{flalign}
from \eqref{phi1} it turns out to be a bijection of $\mathbb{R}^{N\times n}$, see \cite{DVS}. Moreover, from \cite[Lemma 2.10]{DVS} we know that
\begin{flalign}\label{aff40}
V_{\varphi}, \ V_{\varphi}^{-1}\in C^{0,\beta}(\mathbb{R}^{N\times n},\mathbb{R}^{N\times n}) \ \ \mbox{for some} \ \ \beta\in (0,1),
\end{flalign}
with $\beta=\beta(\Delta_{2}(\varphi,\varphi^{*}))$. In the whole paper, \emph{$\varphi$ will always satisfy all the hypotheses listed before}. Our main references for this part are \cite{BCM2, C, DVS}.
\vspace{1mm}
\subsection{Variational setting} In the framework described above, we consider integral functionals of the calculus of variations of the type \eqref{fullf}-\eqref{xf}, where $\Omega \subset \mathbb{R}^{n}$, $n\ge 2$, is an open set, $f\colon \Omega\times \RN\times \mathbb{R}^{N\times n}\to \mathbb{R}$ is a Carath\'eodory integrand such that there exists $\tilde{f}\colon \Omega\times \RN\times [0,\infty)\to \mathbb{R}$ such that
\begin{flalign}\label{assf2}
\begin{cases}
\ f(x,v,z)=\tilde{f}(x,v,\snr{z}) \ \ \mathrm{for \ all \ }(x,v)\in \Omega\times\mathbb{R}^{N}, z \in \mathbb{R}^{N\times n},\\
\ t\mapsto \tilde{f}(\cdot,t)\in C^{1}([0,\infty))\cap C^{2,\vartheta}((0,\infty)) \ \ \mathrm{for \ some} \ \ \vartheta \in (0,1]
\end{cases}
\end{flalign}
and
\begin{flalign}\label{assf}
\begin{cases}
\ z\mapsto f(\cdot,z) \in  C^{1}(\mathbb{R}^{N\times n})\cap C^{2,\vartheta}(\mathbb{R}^{N\times n}\setminus \{0\}),\\
\ x\mapsto f(x,\cdot)\in C^{0,\alpha}(\Omega) \ \ \mbox{for some} \ \ \alpha \in (0,1],\\
\ \nu \varphi(\snr{z})\le f(x,v,z) \le L\varphi(\snr{z}),\\
\ \snr{\partial f(x,v,z)}\snr{z}+\snr{\partial^{2} f(x,v,z)}\snr{z}^{2}\le L\varphi(\snr{z}),\\
\ \nu\varphi'' (\snr{z})\snr{\xi}^{2}\le\langle \partial^{2}f(x,v,z)\xi,\xi\rangle, \\
\ \snr{\partial f(x_{1},v,z)-\partial f(x_{2},v,z)}\snr{z}\le L\tilde{\omega}(\snr{x_{1}-x_{2}})\varphi(\snr{z}),\\
\ \snr{f(x,v_{1},z)-f(x,v_{2},z)}\le L \tilde{\omega}(\snr{v_{1}-v_{2}})\varphi(\snr{z}),\\
\ \snr{\partial^{2}f(x,v,z_{1}+z_{2})-\partial^{2}f(x,v,z_{1})}\le L\varphi''(\snr{z_{1}})\left(\frac{\snr{z_{2}}}{\snr{z_{1}}}\right)^{\vartheta} \ \ \mbox{for} \ \ \snr{z_{2}}<\frac{1}{2}\snr{z_{1}}.
\end{cases}
\end{flalign}
Moreover, we require that $g\colon \Omega\colon\mathbb{R}^{N\times n}\to \mathbb{R}$ is another Carath\'eodory integrand so that 
\begin{flalign}\label{assg}
\begin{cases}
\ z\mapsto g(\cdot,z) \in C^{2}(\mathbb{R}^{N\times n}\setminus \{0\})\cap C^{1}(\mathbb{R}^{N\times n}),\\
\ x\mapsto g(x,\cdot)\in C^{0,\alpha}(\Omega) \ \ \mbox{for some} \ \ \alpha \in (0,1],\\
\ \nu \varphi(\snr{z})\le g(x,z) \le L\varphi(\snr{z})\\
\ \snr{\partial g(x,z)}\snr{z}+\snr{\partial^{2} g(x,z)}\snr{z}^{2}\le L\varphi(\snr{z}),\\
\ \snr{\partial g(x,z_{1})-\partial g(x,z_{2})}\le L\varphi''
(\snr{z_{1}}+\snr{z_{2}})\snr{z_{1}-z_{2}},\\
\ \nu\varphi'' (\snr{z})\snr{\xi}^{2}\le\langle \partial^{2}g(x,z)\xi,\xi\rangle, \\
\ \snr{\partial g(x_{1},z)-\partial g(x_{2},z)}\snr{z}\le L\tilde{\omega}(\snr{x_{1}-x_{2}})\varphi(\snr{z}).
\end{cases}
\end{flalign}
Here $0<\nu< L$ are absolute constants and the above relations hold for any $x\in \Omega$, for all $z,\xi\in \mathbb{R}^{N\times n}$, $v \in \RN$ and $\tilde{\omega}\colon [0,\infty)\to [0,\infty)$ is a continuous, nondecreasing function such that
\begin{flalign}\label{omt}
\tilde{\omega}(r):=\min\left\{r^{\alpha},1\right\} \ \ \mathrm{for \ all \ }r \in (0,1) \ \mathrm{and \ some \ } \alpha \in (0,1].
\end{flalign}
Before proceeding further, let us introduce some other notation. For fixed $(x,v)\in \Omega\times \RN$, we set 
\begin{flalign}\label{aff9}
\tilde{f}_{x,v}(t):=\tilde{f}(x,v,t), \ \ t>0,
\end{flalign}
where $\tilde{f}$ is the map in \eqref{assf2} and, for $z\in \mathbb{R}^{N\times n}$,
\begin{flalign}\label{aff10}
\tilde{V}_{x,v}(z):=\left(\frac{\tilde{f}_{x,v}'(\snr{z})}{\snr{z}}\right)^{\frac{1}{2}}z,
\end{flalign}
which is the vector field defined in \eqref{vphi} with $\varphi=\tilde{f}_{x,v}$. Those positions will be useful to establish extra details on the connection between $\varphi$ and $t \mapsto \tilde{f}(\cdot,t)$, see Lemma \ref{aff} and Corollary \ref{aff1} in Section \ref{pre}.
\vspace{1mm}
\subsection{$\omega$-minima} Let us add some specifics on the concept of $\omega$-minimizer already anticipated in Section \ref{intro}. If $\mathcal{K}(\cdot)$ the integral functional in \eqref{h}, an $\omega$-minimizer of $\mathcal{K}(\cdot)$ is a map $u \in W^{1,\varphi}(\Omega,\RN)$ such that
\begin{flalign*}
\int_{B_{r}}k(x,u,Du) \ dx \le (1+\omega(r))\int_{B_{r}}k(x,w,Dw) \ dx, 
\end{flalign*}
whenever $u-w\in W^{1,\varphi}_{0}(B_{r},\RN)$ and $B_{r}\Subset \Omega$ is any ball with radius $r\le 1$. Here $\omega \colon [0,\infty)\to [0,\infty)$ is a continuous, nondecreasing function so that
\begin{flalign}\label{om}
\omega(r)\le Lr^{\gamma} \ \ \mathrm{for \ all \ }r \in (0,1)
 \ \mathrm{and \ some \ }\gamma \in (0,1],
\end{flalign}
with $L>0$. Adapting the previous definition to \eqref{fullf} and \eqref{xf}, we can conclude that an $\omega$-minimizer of $\mathcal{F}(\cdot)$ satisfies $\int_{B_{r}}f(x,u,Du) \ dx \le (1+\omega(r))\int_{B_{r}}f(x,w,Dw) \ dx$ for all $w\in u+W^{1,\varphi}_{0}(B_{r},\RN)$ and all $B_{r}\Subset \Omega$, while an $\omega$-minimizer of $\mathcal{G}(\cdot)$ matches $\int_{B_{r}}g(x,Du) \ dx \le (1+\omega(r))\int_{B_{r}}g(x,Dw) \ dx$, for any $w \in u+W^{1,\varphi}_{0}(B_{r},\RN)$ and all $B_{r}\Subset \Omega$.

\section{Preliminary results}\label{pre}
In this section we shall collect some well known results in the realm of $N$-functions and of fractional Sobolev spaces.
\vspace{1mm}
\subsection{On $N$-functions}\label{N} As pointed out in Section \ref{assec}, being $\varphi$ convex, the $\Delta_{2}$-condition implies a comparison with power-type functions in the sense that there are two exponents $s_{0}=s_{0}(\Delta_{2}(\varphi))$ and $q_{0}=q_{0}(\Delta_{2}(\varphi))$, $1<s_{0}\le q_{0}$ such that
\begin{flalign}\label{incdec}
t\mapsto \frac{\varphi(t)}{t^{s_{0}}} \ \ \mathrm{is \ non \ decreasing \ }, \quad t\mapsto \frac{\varphi(t)}{t^{q_{0}}} \ \ \mathrm{is \ non \ increasing}.
\end{flalign}
For later uses we point out that $\eqref{incdec}_{1}$ gives, for $t\ge 1$, $\varphi(t)\ge t^{s_{0}}\varphi(1)$ while $\eqref{incdec}_{2}$ renders that, for any $a\in (0,1)$ and all $t\ge 0$ there holds that $\varphi(a^{-1}t)\le a^{-q_{0}}\varphi(t)$. In particular, it is worth stressing that $\eqref{incdec}_{1}$ provides a link to the usual Sobolev space $W^{1,s_{0}}(\Omega,\RN)$: in fact it is easy to see that, if $w \in W^{1,\varphi}(\Omega,\RN)$ there holds
\begin{flalign}\label{incdec1}
\int_{\Omega}\snr{w}^{s_{0}}+\snr{Dw}^{s_{0}} \ dx \le &\int_{\Omega}\max\left\{1,\snr{w}^{s_{0}}\right\}+\max\left\{1,\snr{Dw}^{s_{0}}\right\} \ dx\nonumber \\
\le& c\int_{\Omega}1+\varphi(\snr{w})+\varphi(\snr{Dw}) \ dx,
\end{flalign}
for $c=c(\Delta_{2}(\varphi))$. Such a remark will be helpful in several occasions. We start with a result which is by now standard.
\begin{lemma}\cite{DE}\label{aff20}
Let $\varphi$ be an $N$-function with $\Delta_{2}(\varphi,\varphi^{*})<\infty$. Then, uniformly for all $\snr{z_{1}},\snr{z_{2}}\in \mathbb{R}^{N\times n}$ with $\snr{z_{1}}+\snr{z_{2}}>0$ there holds
\begin{flalign*}
\int_{0}^{1}\frac{\varphi'(\snr{[z_{1},z_{2}]_{\lambda}})}{\snr{[z_{1},z_{2}]_{\lambda}}} \ d\lambda\sim \frac{\varphi'(\snr{z_{1}}+\snr{z_{2}})}{\snr{z_{1}}+\snr{z_{2}}},
\end{flalign*}
with constants depending on $\Delta_{2}(\varphi,\varphi^{*})$.
\end{lemma}
The connections between the vector field $V_{\varphi}$ and the monotonicity properties of the map $\varphi'(\snr{z})/\snr{z}$ are best reflected in the following lemma.
\begin{lemma}\cite{DVS}\label{aff21}
Let $\varphi$ be as described in Subsection \ref{N} and $V_{\varphi}$ be as in \eqref{vphi}. Then
\begin{flalign}\label{phi7}
\left \langle \frac{\varphi'(\snr{z_{1}})}{\snr{z_{1}}}z_{1}-\frac{\varphi'(\snr{z_{2}})}{\snr{z_{2}}}z_{2},z_{1}-z_{2}\right \rangle\sim\snr{V_{\varphi}(z_{1})-V_{\varphi}(z_{2})}^{2}\sim \varphi''(\snr{z_{1}}+\snr{z_{2}})\snr{z_{1}-z_{2}},
\end{flalign}
for all $z_{1},z_{2}\in \mathbb{R}^{N\times n}$ with constants implicit in "$\sim$" depending on $n,N,c_{\varphi}$. Moreover, $\snr{V_{\varphi}(z)}^{2}$ is comparable to $\varphi(\snr{z})$ in the sense that
\begin{flalign}\label{phi9}
\snr{V_{\varphi}(z)}^{2}\sim \varphi(\snr{z}),
\end{flalign}
with constants depending on $\Delta_{2}(\varphi)$.
\end{lemma}
Let us present now an intrinsic variant on the classical \cite[Lemma 6.1]{G}. Even if known, we did not manage to trace it in the literature.
\begin{lemma}\label{L0}
Let $h\colon [\varrho_{0},\varrho_{1}]\to \mathbb{R}$ be nonnegative and bounded and $f\colon [0,\infty)\to [0,\infty)$ be a function such that $f(\lambda t)\le \tilde{c}\lambda^{-\sigma}f(t)$ for all $\lambda \in (0,1]$, all $t \in [0,\infty)$, some $\sigma>0$ and a positive $\tilde{c}$ depending only on the structure of $f$. Assume that
\begin{flalign*}
h(r)\le \theta h(s)+Af(s-r)+B,
\end{flalign*}
holds for some $\theta\in (0,1)$ and all $\varrho_{0}\le r<s\le \varrho_{1}$, where $A,B\ge 0$ are absolute constants. Then the following inequality holds with $c=c(\tilde{c},\theta,\sigma)$:
\begin{flalign*}
h(\varrho_{0})\le c\left(Af(\varrho_{1}-\varrho_{0})+B\right).
\end{flalign*}
\end{lemma}
\begin{proof}
The proof is an easy modification to the one of \cite[Lemma 6.1]{G}, we report it for the sake of completeness. Consider the sequence defined as
\begin{flalign*}\begin{cases}
\ r_{0}=\varrho_{0}\\
\ r_{i+1}-r_{i}=(1-\lambda)\lambda^{i}(\varrho_{1}-\varrho_{0}) \ \ \mathrm{for} \ i \in \N
\end{cases},
\end{flalign*}
where $\lambda \in (0,1)$ will be fixed later on during the proof. By induction and using the properties of $f$, it is easy to see that, for $k \in \N$,
\begin{flalign*}
h(\varrho_{0})\le& \theta^{k}h(r_{k})+\sum_{i=0}^{k-1}\theta^{i}\left[Af((1-\lambda)\lambda^{i}(\varrho_{1}-\varrho_{0}))+B\right]\\
\le &\theta^{k}h(r_{k})+\tilde{c}(1-\lambda)^{-\sigma}\left(f(\varrho_{1}-\varrho_{o})+B\right)\sum_{i=0}^{k-1}\theta^{i}\lambda^{-\sigma i}.
\end{flalign*}
Now fix $\lambda \in (0,1)$ in such a way that $\lambda^{-\sigma}\theta<1$. In correspondence of such a choice we have that the geometric series $\sum_{i=0}^{k-1}\theta^{i}\lambda^{-\sigma i}$ converges and therefore, passing to the limit for $k\to \infty$ we obtain the conclusion for $c(\tilde{c},\theta,\sigma)=\tilde{c}(1-\lambda)^{-\sigma}(1-\theta\lambda^{-\sigma})^{-1}$.
\end{proof}
The following is the by now standard Sobolev-Poincar\'e's inequality for $N$-functions. It is proved in \cite{DE} for the "Poincar\'e-Wirtinger" case, and, after a quick inspection of the proof it is clear that it crucially relies on the estimate $\snr{w-(w)_{B_{\varrho}}}\le c\int_{B_{\varrho}}\frac{\snr{Dw(y)}}{\snr{x-y}^{n-1}} \ dy$, with $c=c(n,N)$, see \cite{MZ}. As shown in \cite[Lemma 7.14]{GT}, a totally analogous estimate holds also if $\left. w \right |_{\partial B_{\varrho}}=0$, so we have both the inequalities as in the classical Sobolev setting. 
\begin{proposition}\cite{DE}\label{P4}
Let $\varphi$ be an $N$-function with $\Delta_{2}(\varphi,\varphi^{*})<\infty$. Then there exists $\theta=\theta(\Delta_{2}(\varphi,\varphi^{*})) \in (0,1)$ and $c=c(n,N,\Delta_{2}(\varphi,\varphi^{*}))$ such that, if $B_{\varrho}\Subset \Omega$, for all $w\in W^{1,\varphi}(B_{\varrho},\RN)$ there holds
\begin{flalign}\label{spa}
\mint_{B_{\varrho}}\varphi\left(\left |\frac{w-(w)_{B_{\varrho}}}{\varrho}\right |\right) \ dx \le c\left(\mint_{B_{\varrho}}\varphi(\snr{Dw})^{\theta} \ dx\right)^{1/\theta}.
\end{flalign}
Moreover, if $w \in W^{1,\varphi}_{0}(B_{\varrho},\RN)$, a similar inequality holds:
\begin{flalign}\label{sp0}
\mint_{B_{\varrho}}\varphi(\snr{w}/\varrho)\ dx \le c\left(\mint_{B_{\varrho}}\varphi(\snr{Dw})^{\theta} \ dx\right)^{1/\theta},
\end{flalign}
where $c$ and $\theta$ have the dependencies specified before.
\end{proposition}
Given all the analogies with the standard growth case, it is natural to expect that minimizers of $\varphi$-functionals enjoy good regularity properties, depending on how close the integrand is to an $N$-function. In this perpective, we have this first result in the realm of fractional Sobolev spaces.
\begin{proposition}\label{P1}
Let $B_{r}\Subset \Omega$ be any ball and $v\in W^{1,\varphi}(B_{r},\RN)$ be a solution to the Dirichlet problem
\begin{flalign}\label{fun1}
u+W^{1,\varphi}_{0}(B_{r},\RN)\ni w \mapsto \min\int_{B_{r}}g(x,Dw) \ dx,
\end{flalign}
where $g$ satisfies \eqref{assg} and $u\in W^{1,\varphi}(\Omega,\RN)$. Then, if $B_{\varrho}\subset B_{r}$ is such that $B_{20\varrho}\Subset B_{r}$ and if $h \in \mathbb{R}^{n}\setminus \{0\}$ with $\snr{h}\le \varrho$, then
\begin{flalign}\label{estv}
\int_{B_{\varrho}}\snr{\tau_{h}V_{\varphi}(Dv)}^{2} \ dx \le c\left(\varrho^{-2}\snr{h}^{2}+\varrho^{-1}\snr{h}^{1+\alpha}+\snr{h}^{2\alpha}\right)\int_{B_{2\varrho}}\snr{V_{\varphi}(Dv)}^{2} \ dx,
\end{flalign}
for $c=c(n,N,\nu,L,\Delta_{2}(\varphi),\Delta_{2}(\varphi,\varphi^{*}))$.
\end{proposition}
\begin{proof}
By minimality, $v$ solves the Euler-Lagrange equation
\begin{flalign*}
\int_{B_{r}}\partial g(x,Dv)\cdot D\xi \ dx=0 \ \ \mathrm{for \ all \ }\xi \in W^{1,\varphi}_{0}(B_{r},\mathbb{R}^{N}),
\end{flalign*}
so, recalling \eqref{assg}, the assumptions of \cite[Theorem 11]{DE} are matched and \eqref{estv} follows.
\end{proof}
The next lemma links assumptions \eqref{assf2} and $\eqref{assf}$. In some sense, it is the Orlicz version of \cite[Lemma 2.12]{acfu}.
\begin{lemma}\label{aff}Let $f$ satisfy assumptions \eqref{assf2}-$\eqref{assf}_{1,3,4,5,8}$. Then the map $\tilde{f}$ introduced in \eqref{assf2} satisfies
\begin{flalign}\label{af}
\begin{cases}\ \nu \varphi(t)\le \tilde{f}(x,v,t)\le L\varphi(t)\\
\ \tilde{f}'(x,v,t)\sim \varphi'(t)\\
\ \tilde{f}''(x,v,t)\sim \varphi''(t)\\
\ \snr{\tilde{f}''(x,v,t+s)-\tilde{f}''(x,v,t)}\le c\left(\frac{\snr{s}}{t}\right)^{\vartheta}\varphi''(t) \ \ \mathrm{for \ all} \ \ \snr{s}<\frac{1}{2}t
\end{cases}
\end{flalign}
for all fixed $(x,v)\in \Omega\times \mathbb{R}^{N}$. The constants implicit in "$\sim$" depend on $\nu,L,\Delta_{2}(\varphi),\Delta_{2}(\varphi,\varphi^{*}),c_{\varphi}$ and $c=c(n,N,L,\Delta_{2}(\varphi),\Delta_{2}(\varphi,\varphi^{*}),c_{\varphi})$. Moreover, there holds
\begin{flalign}\label{aff11}
\snr{V_{\varphi}(z_{1})-V_{\varphi}(z_{2})}^{2}\sim \snr{\tilde{V}_{x,v}(z_{1})-\tilde{V}_{x,v}(z_{2})}^{2},
\end{flalign}
with constants influenced by $n,N,\nu,L,c_{\varphi}$.
\end{lemma}
\begin{proof}
For simplicity, we shall adopt the terminology in \eqref{aff9}. By $\eqref{assf}_{3}$ it immediately follows $\eqref{af}_{1}$. Concerning $\eqref{af}_{2}$, from $\eqref{af}_{1}$ and $\eqref{af2}_{1}$, we easily see that
\begin{flalign}\label{af4}
\snr{\partial f(x,v,z)}=&\left |\ \tilde{f}_{x,v}'(\snr{z})\frac{z}{\snr{z}} \  \right |=\tilde{f}_{x,v}'(\snr{z})\le L\frac{\varphi(\snr{z})}{\snr{z}}\le c\varphi^{'}(\snr{z}),
\end{flalign}
with $c=c(L,\Delta_{2}(\varphi))$. Moreover, a direct computation shows that
\begin{flalign}\label{af1}
\partial^{2}f(x,v,z)=\tilde{f}_{x,v}''(\snr{z})\frac{z\otimes z}{\snr{z}^{2}}+\frac{\tilde{f}_{x,v}'(\snr{z})}{\snr{z}}\left(\mbox{Id}-\frac{z\otimes z}{\snr{z}^{2}}\right).
\end{flalign}
Inserting \eqref{af1} in $\eqref{assf}_{5}$ and choosing $\xi\in \mathbb{R}^{N\times n}\setminus \{0\}$ so that $\xi \cdot z=0$, we obtain
\begin{flalign*}
\sum_{\alpha,\beta=1}^{n}\sum_{i,j=1}^{N}&\left[\tilde{f}_{x,v}''(\snr{z})\frac{(z^{i}_{\alpha}\xi^{i}_{\alpha})(z_{\beta}^{j}\xi^{j}_{\beta})}{\snr{z}^{2}}+\frac{\tilde{f}'_{x,v}(\snr{z})}{\snr{z}}\left(\delta_{\alpha}^{i}\delta_{\beta}^{j}\xi_{\alpha}^{i}\xi_{\beta}^{j}-\frac{(z^{i}_{\alpha}\xi_{\alpha}^{i})(z_{\beta}^{j}\xi_{\beta}^{j})}{\snr{z}^{2}}\right)\right]\nonumber\\
=&\frac{\tilde{f}_{x,v}'(\snr{z})}{\snr{z}}\snr{\xi}^{2}\ge \nu \varphi''(\snr{z})\snr{\xi}^{2},
\end{flalign*}
and, recalling also \eqref{phi6}, we have
\begin{flalign}\label{af3}
\tilde{f}'_{x,v}\snr{z}\ge \nu \varphi''(\snr{z})\snr{z}\ge c\varphi'(\snr{z}),
\end{flalign}
for $c=c(\nu,c_{\varphi})$. Coupling \eqref{af4} and \eqref{af3} we get $\eqref{af}_{2}$. Now, set in $\eqref{assf}_{5}$ $\xi= z\in \mathbb{R}^{N\times n}\setminus \{0\}$. It follows that
\begin{flalign}\label{af5}
\sum_{\alpha,\beta=1}^{n}\sum_{i,j=1}^{N}&\left[\tilde{f}_{x,v}''(\snr{z})\frac{(z_{\alpha}^{i})^{2}(z_{\beta}^{j})^{2}}{\snr{z}^{2}}+\frac{\tilde{f}'_{x,v}(\snr{z})}{\snr{z}}\left(\delta_{\alpha}^{i}\delta_{\beta}^{j}z_{\alpha}^{i}z_{\beta}^{j}-\frac{(z^{i}_{\alpha})^{2}(z_{\beta}^{j})^{2}}{\snr{z}^{2}}\right)\right]\nonumber \\
=&\tilde{f}''_{x,v}(\snr{z})\snr{z}^{2}\ge \nu \varphi''(\snr{z})\snr{z}^{2}.
\end{flalign}
Now, notice that 
\begin{flalign*}
\snr{z}^{-2}(z\otimes z)\cdot \left(\mathrm{Id}-(z\otimes z)\snr{z}^{-2}\right)=0,
\end{flalign*}
therefore, by \eqref{af1} and Phytagoras' Theorem we obtain that
\begin{flalign}\label{af6}
\snr{\partial^{2}f(x,y,z)}^{2}=\snr{\tilde{f}''_{x,v}(\snr{z})}^{2}+\snr{\tilde{f}'_{x,v}(\snr{z})}^{2}(Nn-1)\snr{z}^{-2},
\end{flalign}
therefore, recalling $\eqref{assf}_{4}$ we have
\begin{flalign*}
\snr{\tilde{f}''_{x,v}(\snr{z})}\le L\varphi(\snr{z})\snr{z}^{-2}\le c\varphi''(\snr{z}),
\end{flalign*}
where $c=c(L,c_{\varphi}, \Delta_{2}(\varphi))$. The content of the previous display and \eqref{af5} render $\eqref{af}_{3}$. In \eqref{af6} we also used that $\left | \ \mathrm{Id}-\frac{z\otimes z}{\snr{z}^{2}} \ \right |^{2}=(Nn-1)$. Now, let us take any $z_{1} \in \mathbb{R}^{N\times n}\setminus \{0\}$ and set $t:=\snr{z_{1}}>0$. Let $z_{2} \in \mathbb{R}^{N\times n}$ be any vector parallel to $z_{1}$ and such that $\snr{z_{2}}<\frac{1}{2}\snr{z_{1}}$. Since $z_{2} \parallel z_{1}$ and $\snr{z_{2}}<\frac{1}{2}\snr{z_{1}}$, then there exists $s\in \left(-\frac{t}{2},\frac{t}{2}\right)$ such that $z_{2}=\frac{s}{t}z_{1}$ and $\snr{z_{2}}=\snr{s}$, by the very definition of $t$. Notice that $1+\frac{s}{t}>\frac{1}{2}$, given that $\snr{s}<\frac{1}{2}t$. With this choice of $z_{1}$ and $z_{2}$, keeping in mind \eqref{af1}, $\eqref{assf}_{8}$ reads as
\begin{flalign}\label{af9}
\left | \ \partial^{2}f(x,v,z_{1}+z_{2})\right.&\left.-\partial^{2}f(x,v,z_{1}) \ \right |=\left | \ \tilde{f}''_{x,v}(t+s)\frac{z_{1}\otimes z_{1}}{t^{2}}+ \frac{\tilde{f}_{x,v}'(t+s)}{t+s}\left(\mathrm{Id}-\frac{z_{1}\otimes z_{1}}{t^{2}}\right) \right.\nonumber \\
&\left.\ - \tilde{f}''_{x,v}(t)\frac{z_{1}\otimes z_{1}}{t^{2}}- \frac{\tilde{f}_{x,v}'(t)}{t}\left(\mathrm{Id}-\frac{z_{1}\otimes z_{1}}{t^{2}}\right)\ \right |\nonumber \\
\ge &\snr{\tilde{f}''_{x,v}(t+s)-\tilde{f}_{x,v}''(t)}-(Nn-1)\left | \ \frac{\tilde{f}_{x,v}'(t+s)}{t+s}-\frac{\tilde{f}_{x,v}'(t)}{t}  \ \right |.
\end{flalign}
For $t_{1},t_{2}\in \mathbb{R}$ and $\lambda \in [0,1]$, set $[t_{1},t_{2}]_{\lambda}:=\lambda t_{1}+(1-\lambda)t_{2}$. We point out that, being $\snr{s}<\frac{1}{2}t$, then
\begin{flalign}\label{af7}
[t+s,t]_{\lambda}>\frac{1}{2}t>0.
\end{flalign}
Using \eqref{phi6}, \eqref{af4}, \eqref{af7}, Lemma \ref{aff20}, $\snr{s}<\frac{1}{2}t$, $\eqref{phi1}_{3}$ and the $\Delta_{2}$-condition satisfied by $\varphi$, we estimate
\begin{flalign}\label{af8}
\left | \ \frac{\tilde{f}_{x,v}'(t+s)}{t+s}\right.&\left.-\frac{\tilde{f}_{x,v}'(t)}{t}  \ \right |=\left | \ \int_{0}^{1}\frac{d}{d\lambda}\left(\frac{\tilde{f}'_{x,v}(\snr{[t+s,t]_{\lambda}})}{[t+s,t]_{\lambda}}\right) \ d\lambda \ \right |\nonumber \\
=&\snr{s}\left | \ \int_{0}^{1}\left(\frac{\tilde{f}_{x,v}''([t+s,t]_{\lambda})}{[t+s,t]_{\lambda}}-\frac{\tilde{f}_{x,v}'([t+s,t]_{\lambda})}{[t+s,t]_{\lambda}^{2}}\right) \ d\lambda  \ \right |\nonumber \\
\le& c\snr{s}\int_{0}^{1}\frac{\snr{\tilde{f}'_{x,v}([t+s,t]_{\lambda})}}{[t+s,t]_{\lambda}^{2}} \ d\lambda\le c\frac{\snr{s}}{t}\int_{0}^{1}\frac{\varphi'([t+s,t]_{\lambda})}{[t+s,t]_{\lambda}} \ d\lambda\nonumber \\
\le &c\frac{\snr{s}}{t}\frac{\varphi'(2t+s)}{2t+s}\le c\frac{\snr{s}}{t}\varphi''(t),
\end{flalign}
for $c=c(L,\Delta_{2}(\varphi,\varphi^{*}),c_{\varphi})$. Merging $\eqref{assf}_{8}$, \eqref{af9}, \eqref{af8} and recalling that $\snr{s}/t<\frac{1}{2}$ we obtain
\begin{flalign*}
\left| \ \tilde{f}_{x,v}''(t+s)\right.&\left.- \tilde{f}_{x,v}''(t) \ \right |\le c\frac{\snr{s}}{t}\varphi''(t)+\left | \ \partial^{2}f(x,v,z_{1}+z_{2})\right.\left.-\partial^{2}f(x,v,z_{1}) \ \right |\nonumber \\
\le &c\frac{\snr{s}}{t}\varphi''(t)+L\varphi''(t)\left(\frac{\snr{s}}{t}\right)^{\vartheta}\le c\varphi''(t)\left(\frac{\snr{s}}{t}\right)^{\vartheta},
\end{flalign*}
with $c=c(n,N,L,\Delta_{2}(\varphi,\varphi^{*}),c_{\varphi})$, which is $\eqref{af}_{4}$. Here we also used that $\vartheta \in (0,1]$. Finally, recalling the position in \eqref{aff10}, from $\eqref{af}_{3}$ and Lemma \ref{aff21} applied to both $\tilde{f}_{x,v}$ and $\varphi$ we obtain
\begin{flalign*}
\snr{\tilde{V}_{x,v}(z_{1})-\tilde{V}_{x,v}(z_{2})}^{2}\sim&\left\langle \frac{\tilde{f}_{x,v}'(\snr{z_{1}})}{\snr{z_{1}}}z_{1}- \frac{\tilde{f}_{x,v}'(\snr{z_{2}})}{\snr{z_{2}}}z_{2} \right \rangle\sim \tilde{f}_{x,v}''(\snr{z_{1}}+\snr{z_{2}})\snr{z_{1}-z_{2}}^{2}\nonumber \\
\sim&  \varphi''(\snr{z_{1}}+\snr{z_{2}})\snr{z_{1}-z_{2}}^{2}\sim \snr{V_{\varphi}(z_{1})-V_{\varphi}(z_{2})},
\end{flalign*}
for all $z_{1},z_{2}\subset \mathbb{R}^{N\times n}$, with constants depending at the most on $n,N,\nu,L,c_{\varphi}$, thus \eqref{aff11} is proved.
\end{proof}
\begin{remark}
\emph{We stress that the constants appearing in estimates \eqref{af}-\eqref{aff11}} do not \emph{depend on the choice of $(x,v)\in \Omega\times \RN$}.
\end{remark}
Lemma \ref{aff} has a crucial consequence, as the next corollary shows.
\begin{corollary}\label{aff1}
Let $f$ satisfy \eqref{assf2}-$\eqref{assf}_{1,3,4,5,8}$ and fix $(x,v)\in \Omega\times \mathbb{R}^{N}$. Then, the map $\tilde{f}_{x,v}\colon [0,\infty)\to [0,\infty)$ defined in \eqref{aff9} is an $N$-function matching the assumptions listed in Section \ref{assec}.
\end{corollary}
\begin{proof}
Recall that $\varphi$ satisfies all the assumptions listed in Section \ref{assec}. From \eqref{assf2} and $\eqref{af}_{1,2}$ we directly earn $\eqref{phi1}_{1,2}$. Moreover, $\eqref{af}_{3}$ assures that $\tilde{f}_{x,v}''(t)\ge c\varphi''(t)>0$ for $t>0$ and $c=c(\nu,L,\Delta_{2}(\varphi),\Delta_{2}(\varphi,\varphi^{*}),c_{\varphi})$, so, recalling also that $\tilde{f}_{x,v}'(0)=0$, we can conclude that $t\mapsto \tilde{f}_{x,v}'(t)$ is increasing. By $\eqref{af}_{2}$ we also see that $\tilde{f}_{x,v}'(t)\ge c\varphi'(t)$ with $c=c(\nu,L,\Delta_{2}(\varphi),\Delta_{2}(\varphi,\varphi^{*}),c_{\varphi})$, thus, recalling $\eqref{phi1}_{4}$, $\lim_{t\to \infty}\tilde{f}_{x,v}'(t)=\infty$, therefore \eqref{phi1} is matched. Now we claim that $\tilde{f}_{x,v}$ enjoys the $\Delta_{2}$-condition with $\Delta_{2}(\tilde{f}_{x,v})=\frac{L}{\nu}\Delta_{2}(\varphi)$. In fact, from $\eqref{af}_{1}$ we obtain
\begin{flalign*}
\tilde{f}_{x,v}(2t)\le L\varphi(2t)\le L\Delta_{2}(\varphi)\varphi(t)\le \frac{L}{\nu}\Delta_{2}(\varphi)\tilde{f}_{x,v}(t).
\end{flalign*}
Since from $\eqref{af}_{2}$ it trivially follows that $t\mapsto \tilde{f}_{x,v}(t)$ is increasing, then we just got that $\tilde{f}_{x,v}(t)\sim \tilde{f}_{x,v}(2t)$. Furthermore, from $\eqref{af}_{2,3}$ we get
\begin{flalign}\label{aff4}
\tilde{f}_{x,v}''(t)\sim \varphi''(t)\sim t\varphi'(t)\sim t\tilde{f}_{x,v}'(t),
\end{flalign}
with constants depending on $\nu,L,\Delta_{2}(\varphi),\Delta_{2}(\varphi,\varphi^{*}),c_{\varphi}$. The inequalities in \eqref{af2} can be obtained exploiting $\eqref{af}_{1,2}$ and the definition of $\tilde{f}_{x,v}^{*}$. In fact,
\begin{flalign*}
t\tilde{f}_{x,v}'(t)\sim t\varphi'(t)\sim \varphi(t)\sim \tilde{f}_{x,v}(t),
\end{flalign*}
for all $t\ge 0$, with constants depending on $\nu,L,\Delta_{2}(\varphi),\Delta_{2}(\varphi,\varphi^{*}),c_{\varphi}$ and 
\begin{flalign*}
\tilde{f}^{*}_{x,v}(\tilde{f}_{x,v}'(t))=&\int_{0}^{\tilde{f}_{x,v}'(t)}(\tilde{f}_{x,v}')^{-1}(s) \ ds=\int_{0}^{(\tilde{f}_{x,v}')^{-1}(\tilde{f}_{x,v}'(t))}(\tilde{f}_{x,v}')^{-1}(\tilde{f}'_{x,v}(s))\tilde{f}_{x,v}''(s) \ ds\nonumber \\
=&\int_{0}^{t}s\tilde{f}_{x,v}''(s) \ ds\stackrel{\eqref{aff4}}{\sim} \int_{0}^{t}\tilde{f}_{x,v}'(s) \ dx =\tilde{f}_{x,v}(t). 
\end{flalign*}
Here, the constants implicit in "$\sim$" depend on $\nu,L,\Delta_{2}(\varphi),\Delta_{2}(\varphi,\varphi^{*}),c_{\varphi}$. Finally, $\eqref{af}_{4}$ gives directly \eqref{aff5}. 
\end{proof}
Corollary \eqref{aff1} combined with the theory exposed in \cite{DVS} renders the following, important reference estimates for frozen functionals concerning the gradient of solutions and the decay of the associated excess functional.
\begin{proposition}\label{aff30}
Under assumptions \eqref{assf}-$\eqref{assf}_{1,3,4,5,8}$, fix any couple $(x,v)\in \Omega\times \RN$ and let $\tilde{f}_{x,v}\colon [0,\infty)\to [0,\infty)$ be the map defined in \eqref{aff9} and $B_{r}\Subset \Omega$ be any ball. Then, if $v\in W^{1,\varphi}(B_{r},\RN)$ solves the Dirichlet problem
\begin{flalign*}
u+W^{1,\varphi}_{0}(B_{r},\RN)\ni w\mapsto \int_{B_{r}}\tilde{f}_{x,v}(\snr{Dw}) \ dx,
\end{flalign*}
for some $u \in W^{1,\varphi}(\Omega,\RN)$, then for every ball $B_{t}\subset B_{s}\Subset B_{r}$ there holds
\begin{flalign}\label{aff6}
\sup_{B_{t}}\varphi(\snr{Dv})\le c\mint_{B_{s}}\varphi(\snr{Dv}) \ dx,
\end{flalign}
with $c=c(n,N,\nu,L,\Delta_{2}(\varphi),\Delta_{2}(\varphi,\varphi^{*}),c_{\varphi})$. Moreover, for any $B_{t}\subset B_{s}\Subset B_{r}$ there holds
\begin{flalign}\label{aff7}
\mint_{B_{t}}\snr{V_{\varphi}(Dv)-(V_{\varphi}(Dv))_{B_{t}}}^{2} \ dx \le c(t/s)^{\tilde{\nu}}\mint_{B_{s}}\varphi(\snr{Dv}) \ dx,
\end{flalign}
for $\tilde{\nu}=\tilde{\nu}(n,N,\nu,L,\Delta_{2}(\varphi),\Delta_{2}(\varphi,\varphi^{*}),c_{\varphi})$ and $c=c(n,N,\nu,L,\Delta_{2}(\varphi),\Delta_{2}(\varphi,\varphi^{*}),c_{\varphi})$.
\end{proposition}
\begin{proof}
From Corollary \ref{aff1} and \eqref{assf2}-\eqref{assf}, we see that $\tilde{f}_{x,v}$ satisfies all the assumptions required by \cite[Theorem 1.1]{DVS}, thus \eqref{aff6}-\eqref{aff7} readily follow.
\end{proof}
\subsection{On fractional Sobolev spaces} The following lemma is the key to show that a certain function belongs to the fractional Sobolev space $W^{\sigma,p}_{\mathrm{loc}}(\Omega,\RN)$.
\begin{lemma}\label{frem}\cite{KM1}
Let $\tilde{\Omega}\Subset \Omega$ be an open subset and $G\in L^{p}(\Omega,\RN)$ is any vector field such that
\begin{flalign*}
\int_{\tilde{\Omega}}\snr{\tau_{h}G(x)}^{p} \ dx \le M^{p}\snr{h}^{p\tilde{\sigma}},
\end{flalign*}
for every $h \in \mathbb{R}^{n}\setminus \{0\}$ so that $0<\snr{h}\le \min\left\{1,\dist(\tilde{\Omega},\partial \Omega)\right\}$, with $M\ge 0$ and $\tilde{\sigma}\in (0,1)$. Then $G\in W^{\sigma,p}_{\mathrm{loc}}(\tilde{\Omega},\mathbb{R}^{N\times n})$ for all $\sigma \in (0,\tilde{\sigma})$. Moreover, for each open set $\Omega_{0}\Subset \tilde{\Omega}$, there exists a constant $c=c(\sigma,p,\dist(\Omega_{0},\partial \tilde{\Omega}))$ independent of $M$ and $G$ such that
\begin{flalign*}
\nr{G}_{W^{\sigma,p}(\Omega_{0},\mathbb{R}^{N\times n})}\le c\left(M+\nr{G}_{L^{p}(\Omega,\mathbb{R}^{N\times n})}\right).
\end{flalign*}
\end{lemma}
\begin{remark}\label{R1}
\emph{The constant $c$ in the previous Lemma depends on $\dist(\tilde{\Omega},\partial \Omega)$ and on $\tilde{\sigma}$. It becomes unbounded when $\dist(\tilde{\Omega},\partial \Omega)\to 0$ or $\tilde{\sigma}\to 0$, see \cite{KM1}}.
\end{remark}
We conclude this section with the usual "fractional trading", i.e. the fact that a map having a fractional derivative in the right Lebesgue space results in the earning of some extra integrability for the function itself.
\begin{lemma}\label{intem}\cite{DPV}
Let $\sigma\in (0,1)$, $p\in [1,\infty)$ be such that $\sigma p<n$ and $\Omega\subset \mathbb{R}^{n}$ be an extension domain for $W^{\sigma,p}$. Then there exists a positive constant $c=c(n,N,p,b,\Omega)$ such that for any $w \in W^{b,p}(\Omega,\RN)$ we have
\begin{flalign*}
\nr{w}_{L^{q}(\Omega,\RN)}\le c\nr{w}_{W^{\sigma,p}(\Omega,\RN)} \ \ \mathrm{for \ all \ }q \in \left[p,\frac{np}{n-\sigma p}\right].
\end{flalign*}
In particular, if $\Omega$ is bounded, the previous inequality holds for all $q \in \left[1,\frac{np}{n-\sigma p}\right]$. 
\end{lemma}
Following \cite{DPV}, we will take as an extension domain for $W^{\sigma,p}$ any open set with bounded Lipschitz boundary. Since our results are local in nature, we will mostly work on balls and then fillet the resulting estimates via covering arguments.

\section{Basic regularity results}\label{basic}
We start with a fundamental tool in regularity, the celebrated Caccioppoli's inequality.
\begin{lemma}\label{L4}
Let $0< \varrho\le 1$ and $B_{2\varrho}\Subset \Omega$ be any ball. Then if $u \in W^{1,\varphi}(\Omega,\RN)$ is an $\omega$-minimizer of \eqref{fullf}, with $f$ satisfying $\eqref{assf}_{3}$, there exists $c=c(n,N,\nu,L,\Delta_{2}(\varphi))>0$ such that
\begin{flalign*}
\mint_{B_{\varrho}}\varphi(\snr{Du}) \ dx \le c\mint_{B_{2\varrho}}\varphi\left(\left |\frac{u-(u)_{B_{2\varrho}}}{\varrho}\right |\right) \ dx.
\end{flalign*}
\end{lemma}
\begin{proof}
Consider parameters $0<\varrho\le t<s\le 2\varrho$ and select $\eta \in C^{1}_{c}(B_{s})$ such that $\chi_{B_{t}}\le \eta\le \chi_{B_{s}}$ and $\snr{D\eta}\le (s-t)^{-1}$. The map $v:=u-\eta(u-(u)_{B_{2\varrho}})$ is an admissible competitor for $u$ over $B_{s}$, so, by definition of $\omega$-minimality and $\eqref{assf}_{3}$ we have
\begin{flalign*}
\int_{B_{t}}\varphi(\snr{Du}) \ dx \le c\int_{B_{s}\setminus B_{t}}\varphi(\snr{Du}) \ dx +c\int_{B_{2r}}\varphi\left(\left |\frac{u-(u)_{B_{2\varrho}}}{s-t} \right |\right) \ dx,
\end{flalign*}
for $c=c(\nu,L,\Delta_{2}(\varphi))$. Summing on both sides of the inequality in the previous display the quantity $c\int_{B_{t}}\varphi(\snr{Du}) \ dx$ we end up with \begin{flalign}\label{2}
\int_{B_{t}}\varphi(\snr{Du}) \ dx \le \frac{c}{1+c}\int_{B_{s}}\varphi(\snr{Du}) \ dx +c \int_{B_{\varrho}}\varphi\left(\left |\frac{u-(u)_{B_{\varrho}}}{s-t}\right|\right) \ dx,
\end{flalign}
where $c$ has the dependencies outlined before. Notice that the choice $$f(t)=\int_{B_{\varrho}}\varphi\left(\left |\frac{u-(u)_{B_{\varrho}}}{t} \right |\right) \ dx,$$ is admissible for an application of Lemma \ref{L0} in the light of the discussion at the beginning of Section \ref{pre}, since
\begin{flalign*}
f(\lambda t)=\int_{B_{\varrho}}\varphi\left(\left |\frac{u-(u)_{B_{\varrho}}}{\lambda t} \right |\right) \ dx\le \tilde{c}\lambda^{-q_{0}}\int_{B_{\varrho}}\varphi\left(\left |\frac{u-(u)_{B_{\varrho}}}{t} \right |\right) \ dx=\tilde{c}\lambda^{-q_{0}}f(t).
\end{flalign*}
Here $\tilde{c}=\tilde{c}(\Delta_{2}(\varphi))$. From Lemma \ref{L0} we obtain:
\begin{flalign}\label{1}
\int_{B_{\varrho/2}}\varphi(\snr{Du}) \ dx \le c\int_{B_{\varrho}}\varphi\left(\left |\frac{u-(u)_{B_{\varrho}}}{\varrho} \right |\right) \ dx,
\end{flalign}
with $c=c(n,N,\nu,L,\Delta_{2}(\varphi))$.
\end{proof}
Combining Caccioppoli's inequality with Proposition \ref{P4} we obtain inner higher integrability of Gehring type.
\begin{lemma}\label{L1}
Let $u \in W^{1,\varphi}(\Omega,\RN)$ be an $\omega$-minimizer of \eqref{fullf}, with $f$ satisfying assumptions  $\eqref{assf}_{3}$. Then there exists a positive $\delta_{g}=\delta_{g}(n,N,\nu,L,\Delta_{2}(\varphi),\Delta_{2}(\varphi,\varphi^{*}),\gamma)$ and a constant $c=c(n,N,\nu,L,\Delta_{2}(\varphi),\Delta_{2}(\varphi,\varphi^{*}),\gamma)$ such that, if $B_{2\varrho}\Subset \Omega$ is any ball with $\varrho\le 1$, there holds
\begin{flalign*}
\left(\mint_{B_{\varrho}}\varphi(\snr{Du})^{1+\sigma} \ dx\right)^{\frac{1}{1+\sigma}}\le c\mint_{B_{2\varrho}}\varphi(\snr{Du}) \ dx, \ \ \mathrm{for \ all \ }\sigma \in [0,\delta_{g}].
\end{flalign*}
\end{lemma}
\begin{proof}
Let us fix $B_{2\varrho}\Subset \Omega$, $0<\varrho\le 1$. If $u \in W^{1,\varphi}(\Omega,\RN)$ is an $\omega$-minimizer of \eqref{fullf}, then Lemma \ref{L4} and \eqref{spa} apply, thus rendering
\begin{flalign*}
\mint_{B_{\varrho}}\varphi(\snr{Du}) \ dx \le c \left(\mint_{B_{2\varrho}}\varphi(\snr{Du})^{\theta} \ dx\right)^{1/\theta},
\end{flalign*}
for some $\theta=\theta(\Delta_{2}(\varphi,\varphi^{*})) \in (0,1)$ and $c=c(n,N,\nu,L,\Delta_{2}(\varphi),\Delta_{2}(\varphi,\varphi^{*}),\gamma)$. Now the conclusion easily follows from a variant of Gehring's Lemma, see \cite[Chapter 6]{G}.
\end{proof}
The above result can be carried up to the boundary, as the next lemma shows.
\begin{lemma}\label{L2}
Let $B_{r}\Subset \Omega$ be any ball and $v \in W^{1,\varphi}(B_{r},\RN)$ be a solution to the Dirichlet problem
\begin{flalign*}
u+W^{1,\varphi}_{0}(B_{r},\RN)\ni w \mapsto \min \int_{B_{r}}f(x,w,Dw) \ dx,
\end{flalign*}
where $f$ satisfies $\eqref{assf}_{3}$ and $\varphi(\snr{Du})\in L^{1+\delta}_{\mathrm{loc}}(\Omega)$ for some $\delta>0$. Then there exists a $\sigma_{g}\in (0,\delta)$ such that
\begin{flalign*}
\left(\mint_{B_{r}}\varphi(\snr{Dv})^{1+\sigma} \ dx \right)^{\frac{1}{1+\sigma}}\le c \left(\mint_{B_{r}}\varphi(\snr{Du})^{1+\sigma} \ dx\right)^{\frac{1}{1+\sigma}}, \ \ \mbox{for all} \ \ \sigma \in [0,\sigma_{g}].
\end{flalign*}
Here $\sigma_{g}=\sigma_{g}(n,N,\nu,L,\Delta_{2}(\varphi),\Delta_{2}(\varphi,\varphi^{*}))$ and $c=c(n,N,\nu,L,\Delta_{2}(\varphi),\Delta_{2}(\varphi,\varphi^{*}))$.
\end{lemma}
\begin{proof}
With $x_{0}\in B_{r}$, let us fix a ball $B_{\varrho}(x_{0})\subset \mathbb{R}^{n}$, $\varrho\le 1$. We consider first the case in which $\snr{B_{\varrho}(x_{0})\setminus B_{r}}\ge \snr{B_{\varrho}(x_{0})}/8$. For $0<\varrho/2\le t<s\le \varrho$, we take $\eta \in C^{1}_{c}(B_{s}(x_{0}))$ such that $\chi_{B_{t}(x_{0})}\le \eta\le \chi_{B_{s}(x_{0})}$, $\snr{D\eta}\le (s-t)^{-1}$ and define $\psi:=v-\eta^{2}(v-u)$. The minimality of $v$ and our choice of $\eta$ render
\begin{flalign*}
\int_{B_{t}(x_{0})\cap B_{r}}\varphi(\snr{Dv}) \ dx \le &c\int_{(B_{s}(x_{0})\setminus B_{t}(x_{0}))\cap B_{r}}\varphi(\snr{Dv}) \ dx\nonumber \\
&+c\int_{B_{\varrho}(x_{0})\cap B_{r}}\varphi(\snr{Du})+\varphi\left(\left |\frac{u-v}{s-t} \right|\right) \ dx,
\end{flalign*}
with $c=c(n,N,\nu,L,\Delta_{2}(\varphi))$, and summing on both sides of the above inequality the quantity $$\int_{B_{t}(x_{0})\cap B_{r}}\varphi(\snr{Dv}) \ dx$$
and applying Lemma \ref{L0} as we did in the proof of Lemma \ref{L4} we obtain
\begin{flalign*}
\int_{B_{\varrho/2}(x_{0})\cap B_{r}}\varphi(\snr{Dv}) \ dx \le c\int_{B_{\varrho}(x_{0})\cap B_{r}}\varphi(\snr{Du})+\varphi\left(\left |\frac{u-v}{\varrho} \right|\right) \ dx,
\end{flalign*}
with $c=c(n,N,\nu,L,\Delta_{2}(\varphi))$. Applying \eqref{sp0}, after standard manipulations we can conclude that
\begin{flalign*}
\mint_{B_{\varrho/2}(x_{0})\cap B_{r}}\varphi(\snr{Dv}) \ dx \le c\left\{\mint_{B_{\varrho}(x_{0})\cap B_{r}}\varphi(\snr{Du}) \ dx +\left(\mint_{B_{\varrho(x_{0})\cap B_{r}}}\varphi(\snr{Dv})^{\theta} \ dx\right)^{1/\theta}\right\}, 
\end{flalign*}
with $c=c(n,N,\nu,L,\Delta_{2}(\varphi), \Delta_{2}(\varphi,\varphi^{*}))$. We next consider the situation when it is $B_\varrho \Subset B_r$. In this case we have the usual Sobolev-Poincar\'e's inequality as in the interior case and there is no loss of generality in taking the same exponent. The two cases can be combined via a standard covering argument. More precisely, upon defining 
\begin{flalign*}
V(x):=\begin{cases}
\ \varphi(\snr{Dv})^{\theta} \quad &x\in B_{\varrho}(x_{0})\\
\ 0 \quad &x\in \mathbb{R}^{n}\setminus B_{\varrho}(x_{0})
\end{cases}\quad and \quad U(x):=\begin{cases}
\ \varphi(\snr{Du}) \quad &x\in B_{\varrho}(x_{0})\\
\ 0 \quad &x\in \mathbb{R}^{n}\setminus B_{\varrho}(x_{0})
\end{cases}
\end{flalign*}
we get
\begin{flalign*}
\mint_{B_{\varrho/2}(x_{0})}[V(x)]^{1/\theta} \, dx\le c\left \{ \left(\mint_{B_{\varrho}(x_{0})}V(x) \, dx\right)^{1/\theta}+\mint_{B_{\varrho}(x_{0})}U(x) \, dx \right\},
\end{flalign*}
with $c\equiv c(n,N,\nu,L,\Delta_{2}(\varphi),\Delta_{2}(\varphi,\varphi^{*}))$ and $0<\theta<1$. At this point the conclusion follows by the minimality of $v$ and a standard variant of Gehring's lemma. 
\end{proof}

\section{Proof of Theorem \ref{frac}}\label{te1}
To show this result, we use the "Variational difference quotient technique" developed in the contest of autonomous functionals with standard $p$-growth in \cite{KM1}. For simplicity, we split the proof in four steps.\\\\
\emph{Step 1: introducing a scale}. Let $\Omega_{0}\Subset \Omega$ and $\beta \in (0,1)$, which value will be fixed later on in the proof. We consider vectors $h \in \mathbb{R}^{n}\setminus \{0\}$ such that
\begin{flalign}\label{3}
0<\snr{h}\le \min\left\{\left(\frac{\dist(\Omega_{0},\partial \Omega)}{1000\sqrt{n}}\right)^{\frac{1}{\beta}},\left(\frac{1}{1000}\right)^{\frac{1}{1-\beta}}\right\}.
\end{flalign}
For such values of $\snr{h}$ and $x_{0}\in \Omega$ with $\dist(x_{0},\Omega_{0})<\snr{h}^{\beta}$ we define $B(h):=B_{\snr{h}^{\beta}}(x_{0})$. With the above restriction on $\snr{h}$, we have in particular that $B_{10\sqrt{n}\snr{h}^{\beta}}(x_{0})\Subset \Omega$.
\\\\
\emph{Step 2: a comparison map}. Let $v \in u+W^{1,\varphi}_{0}(B(h),\RN)$ be the unique minimizer of the functional
\begin{flalign*}
u+W^{1,\varphi}_{0}(B(h),\RN)\ni w \mapsto \min \int_{B(h)}g(x,Dw) \ dx. 
\end{flalign*}
Since $g$ satisfies $\eqref{assg}$, existence and uniqueness follow by Direct methods. The minimality of $v$ and $\eqref{assg}_{3}$ yield that
\begin{flalign}\label{4}
\begin{cases}
\ \int_{B(h)}g(x,Dv) \ dx \le \int_{B(h)}g(x,Du) \ dx \le L\int_{B_{h}}\varphi(\snr{Du}) \ dx,\\
\ \int_{B(h)}\varphi(\snr{Dv}) \ dx \le \frac{L}{\nu}\int_{B(h)}\varphi(\snr{Du}) \ dx.
\end{cases}
\end{flalign}
The $\omega$-minimality of $u$ renders that
\begin{flalign*}
\int_{B(h)}g(x,Du) \ dx \le (1+\omega(\snr{h}^{\beta}))\int_{B(h)}g(x,Dv) \ dx,
\end{flalign*}
so, recalling also $\eqref{4}_{1}$,
\begin{flalign}\label{5}
\int_{B(h)}g(x,Du)-g(x,Dv) \ dx \le L\omega(\snr{h}^{\beta})\int_{B(h)}\varphi(\snr{Du}) \ dx.
\end{flalign}
From $\eqref{assg}_{6}$, the minimality of $v$ and \eqref{5} we get
\begin{flalign}\label{6}
c\int_{B(h)}&\varphi''(\snr{Du}+\snr{Dv})\snr{Du-Dv}^{2} \ dx \nonumber \\
=&c\int_{B(h)}\varphi''(\snr{Du}+\snr{Dv})\snr{Du-Dv}^{2} \ dx+c\int_{B(h)}\partial g(x,Dv)(Du-Dv) \ dx\nonumber  \\
\le &\int_{B(h)}g(x,Du)-g(x,Dv) \ dx\le L\omega(\snr{h}^{\beta})\int_{B(h)}\varphi(\snr{Du}) \ dx.
\end{flalign}
for $c=c(n,N,\nu,L)$. Combining \eqref{6} with $\eqref{phi7}$ we obtain
\begin{flalign}\label{7}
\int_{B(h)}\snr{V_{\varphi}(Du)-V_{\varphi}(Dv)}^{2} \ dx \le c\omega(\snr{h}^{\beta})\int_{B(h)}\varphi(\snr{Du}) \ dx,
\end{flalign}
with $c=c(n,N,\nu,L,c_{\varphi})$.\\\\
\emph{Step 3: a fractional estimate for $u$}. From \eqref{phi9}, \eqref{om}, \eqref{estv}, $\eqref{4}_{2}$ and \eqref{7} we estimate
\begin{flalign}\label{8}
\int_{B_{\snr{h}^{\beta}/40}(x_{0})}\snr{\tau_{h}V_{\varphi}(Du)}^{2} \ dx \le &c \int_{B_{\snr{h}^{\beta}/40}(x_{0})}\snr{\tau_{h}V_{\varphi}(Dv)}^{2} \ dx+c\int_{B(h)}\snr{V_{\varphi}(Du)-V_{\varphi}(Dv)}^{2}\ dx\nonumber \\
\le & c\left(\snr{h}^{2-2\beta}+\snr{h}^{1+\alpha-\beta}+\snr{h}^{2\alpha}+\snr{h}^{\beta\gamma}\right)\int_{B(h)}\varphi(\snr{Du}) \ dx\nonumber \\
\le& c\snr{h}^{2\sigma}\int_{B(h)}\varphi(\snr{Du}) \ dx, 
\end{flalign}
with $c=c(n,N,\nu,L,\Delta_{2}(\varphi),\Delta_{2}(\varphi,\varphi^{*}))$,
\begin{flalign*}
\beta:=\min\left\{ \frac{2}{2+\gamma},\frac{1+\alpha}{1+\gamma} \right\} \ \ \mbox{and} \ \ \sigma:=\min\left\{\alpha,\frac{\gamma}{2+\gamma},\frac{\gamma(1+\alpha)}{2(1+\gamma)}\right\}.
\end{flalign*}
\emph{Step 4: a final covering argument}. We use the inclusion
\begin{flalign*}
Q_{\snr{h}^{\beta}/40\sqrt{n}}(x_{0})\subset B_{\snr{h}^{\beta}/40}(x_{0}) \quad \mathrm{and}\quad B_{\snr{h}^{\beta}}(x_{0})\subset Q_{\snr{h}^{\beta}}(x_{0}),
\end{flalign*}
to get from \eqref{8},
\begin{flalign}\label{9}
\int_{Q_{\snr{h}^{\beta}/40\sqrt{n}}(x_{0})}\snr{\tau_{h}V_{\varphi}(Du)}^{2} \ dx \le c\snr{h}^{2\sigma}\int_{Q_{\snr{h}^{\beta}}(x_{0})}\varphi(\snr{Du}) \ dx.
\end{flalign}
Fix a vector $h \in \mathbb{R}^{n}\setminus \{0\}$ such that $\snr{h}$ satisfies \eqref{3}. Notice that $x_{0}$ is any point in $\Omega$ such that $\dist(x_{0},\Omega_{0})<\snr{h}^{\beta}$, so we can find a finite family of disjoint cubes $\{Q^{i}\}_{i=1}^{K}$, $Q^{i}=Q^{i}_{\snr{h}^{\beta}/40\sqrt{n}}(x_{i})$, with $K=K(\Omega_{0},\snr{h})\in \N$ such that $\Omega_{0}\subset \bigcup_{i=1}^{K} \overline{Q^{i}}$, thus estimate \eqref{9} is valid with $x_{0}\equiv x_{i}$ for all $i \in \{1,\cdots,K\}$. Notice that, since the $Q^{i}$'s are disjoint, each of the dilated cubes $40\sqrt{n}Q^{i}$ intersects at the most $(80\sqrt{n})^{n}$ of the other dilated cubes $40\sqrt{n}Q^{j}$. Furthermore, in view of \eqref{3} we also know that $40\sqrt{n}Q^{i}\subset B_{\sqrt{n}\snr{h}^{\beta}}(x_{i})\subset \Omega$. Hence, if we take $x_{0}\equiv x_{i}$ in \eqref{9} and sum the resulting inequalities over $i \in \{1,\cdots,K\}$ we then have
\begin{flalign*}
\int_{\Omega_{0}}\snr{\tau_{h}V_{\varphi}(Du)}^{2} \ dx \le c\snr{h}^{2\sigma}\int_{\Omega}\varphi(\snr{Du}) \ dx,
\end{flalign*}
with $c=c(n,N,\nu,L,\Delta_{2}(\varphi),\Delta_{2}(\varphi,\varphi^{*}),\alpha,\gamma, \Omega_{0})$ independent on $\snr{h}$. Now we can apply Lemma \ref{frem} to conclude, after a standard covering argument that
\begin{flalign}\label{fffrac}
V_{\varphi}(Du)\in W^{\delta,2}_{\mathrm{loc}}(\Omega,\mathbb{R}^{N\times n}) \ \ \mathrm{for \ all \ } \delta \in (0,\sigma).
\end{flalign}
Notice that any given ball $B_{r}\Subset \Omega$ is an extension domain for $W^{\delta,2}$, so, from \eqref{fffrac} and Lemma \ref{intem}, we obtain that $V_{\varphi}(Du)\in L^{\frac{2n}{n-2\delta}}(B_{r},\mathbb{R}^{N\times n})$ for all $\delta \in (0,\sigma)$. Again after covering, we recover \eqref{res}. 

\section{Proof of Theorem \ref{reg}}\label{te2}
For the reader's convenience, we frame this proof into five steps. Precisely, in the first one we show an intrinsic decay estimate for $\int\varphi(\snr{Du}) \ dx$, which in turn implies the $\beta$-H\"older continuity of $u$ for any $\beta \in (0,1)$. Then we look at the structure of the singular set $\Sigma_{u}$ and use the characteristics of $\varphi$ and the inner higher integrability result in Lemma \ref{L1} to obtain an upper bound on $\dim_{\mathcal{H}}(\Sigma_{u})$. Finally, a straightforward manipulation of the estimates obtained so far renders the local H\"older continuity of $V_{\varphi}(Du)$.\\\\ 
\emph{Step 1: intrinsic Morrey decay}. We start by assuming that $1<s_{0}(1+\delta_{g})\le n$, where $\delta_{g}$ is the higher integrability threshold provided by Lemma \ref{L1} and $s_{0}$ is the exponent of the power function controlled (up to constants) from the above by $\varphi$. In case $s_{0}(1+\delta_{g})>n$, by \eqref{incdec1} and Morrey's embedding theorem we get that $u \in C^{0,\lambda}_{\mathrm{loc}}(\Omega,\RN)$, with $\lambda:=1-\frac{n}{s_{0}(1+\delta_{g})}$, and the procedure is slightly different, this case will be treated in \emph{Step 3}. Let $B_{r}=B_{r}(x_{0})$, $r\le 1/2$, be any ball such that $B_{2r}\Subset \Omega$ and assume that the smallness condition
\begin{flalign}\label{small}
E(B_{2r}):=\mint_{B_{2r}}\varphi(\snr{Du}) \ dx <\varphi\left(\frac{\varepsilon}{r}\right)
\end{flalign}
holds. We introduce the comparison map $v\in W^{1,\varphi}(B_{r},\RN)$ defined as a solution to the Dirichlet problem
\begin{flalign}\label{fz}
u+W^{1,\varphi}_{0}(B_{r},\RN)\ni w\mapsto \min\int_{B_{r}}f(x_{0},(u)_{B_{r}},Dw) \ dx.
\end{flalign}
Notice that, by $\eqref{assf}_{3}$ and the minimality of $v$, we have
\begin{flalign}\label{11}
\mint_{B_{r}}\varphi(\snr{Dv}) \ dx \le& \nu^{-1}\mint_{B_{r}}f(x_{0},(u)_{B_{r}},Dv) \ dx\nonumber \\
\le& \nu^{-1}\mint_{B_{r}}f(x_{0},(u)_{B_{r}},Du) \ dx \le \frac{L}{\nu}\mint_{B_{r}}\varphi(\snr{Du}) \ dx.
\end{flalign}
We recall the strict monotonicity property 
\begin{flalign}\label{sm}
\snr{V_{\varphi}(z_{1})-V_{\varphi}(z_{2})}^{2}&+c\langle\partial f(x_{0},(u)_{B_{r}},z_{1}),z_{2}-z_{1}\rangle\nonumber \\
\le &c\left(f(x_{0},(u)_{B_{r}},z_{2})-f(x_{0},(u)_{B_{r}},z_{1})\right),
\end{flalign}
for $c=c(n,N,\nu,c_{\varphi})$, so, from the minimality of $v$ and \eqref{sm} we obtain that
\begin{flalign*}
\mint_{B_{r}}&\snr{V_{\varphi}(Du)-V_{\varphi}(Dv)}^{2} \ dx \le c\mint_{B_{r}}f(x_{0},(u)_{B_{r}},Du)-f(x_{0},(u)_{B_{r}},Dv) \ dx \\
=&c\mint_{B_{r}}f(x_{0},(u)_{B_{r}},Du)-f(x_{0},u,Du) \ dx +c\mint_{B_{r}}f(x_{0},u,Du)-f(x,u,Du) \ dx \\
&+c\mint_{B_{r}}f(x,u,Du)-f(x,v,Dv) \ dx+c\mint_{B_{r}}f(x,v,Dv)-f(x,(v)_{B_{r}},Dv) \ dx \\
&+c\mint_{B_{r}}f(x,(v)_{B_{r}},Dv)-f(x_{0},(v)_{B_{r}},Dv) \ dx\nonumber \\
&+c\mint_{B_{r}}f(x_{0},(v)_{B_{r}},Dv)-f(x_{0},(u)_{B_{r}},Dv) \ dx=:\mathrm{(I)}+\mathrm{(II)}+\mathrm{(III)}+\mathrm{(IV)}+\mathrm{(V)}+\mathrm{(VI)},
\end{flalign*}
with $c=c(n,N,\nu,c_{\varphi})$. Before start working on terms $\mathrm{(I)}-\mathrm{(VI)}$, let us consider some quantities which will be recurrent in the forthcoming estimates. From $\eqref{omt}$, Jensen's inequalities (for both concave and convex functions), \eqref{spa} and \eqref{small} we get
\begin{flalign}\label{omt1}
\mint_{B_{r}}\tilde{\omega}(\snr{u-(u)_{B_{r}}}) \ dx \le &\tilde{\omega}\left(r\varphi^{-1}\circ\varphi\left(\mint_{B_{r}}\left |\frac{u-(u)_{B_{r}}}{r} \right | \ dx\right)\right)\nonumber \\
\le&\tilde{\omega}\left(r\varphi^{-1}\left(\mint_{B_{r}}\varphi\left(\left |\frac{u-(u)_{B_{r}}}{r} \right |\right) \ dx\right)\right)\nonumber \\
\le &c\tilde{\omega}\left(r\varphi^{-1}\left(\mint_{B_{r}}\varphi(\snr{Du}) \ dx\right)\right)\le c\tilde{\omega}\left(r\varphi^{-1}\left(\varphi(\varepsilon/r)\right)\right)\le c\varepsilon^{\alpha},
\end{flalign}
where $c=c(n,N,\Delta_{2}(\varphi),\Delta_{2}(\varphi,\varphi^{*}))$. In a totally similar way, but using this time \eqref{11}, we obtain
\begin{flalign}\label{omt2}
\mint_{B_{r}}\tilde{\omega}(\snr{v-(v)_{B_{r}}}) \ dx \le &\tilde{\omega}\left(r\varphi^{-1}\left(\mint_{B_{r}}\varphi(\snr{Dv}) \ dx\right)\right)\nonumber \\
\le& c\tilde{\omega}\left(r\varphi^{-1}\left(\mint_{B_{r}}\varphi(\snr{Du}) \ dx\right)\right)\le c\varepsilon^{\alpha},
\end{flalign}
for $c=c(n,N,\nu,L,\Delta_{2}(\varphi),\Delta_{2}(\varphi,\varphi^{*}))$. Finally, using \eqref{sp0} and again \eqref{11} we have
\begin{flalign}\label{omt3}
\mint_{B_{r}}\tilde{\omega}(\snr{u-v}) \ dx \le &c\tilde{\omega}\left(r\varphi^{-1}\left(\mint_{B_{r}}\varphi(\snr{Du-Dv}) \ dx\right)\right)\le c\varepsilon^{\alpha},
\end{flalign}
with $c=c(n,N,\nu,L,\Delta_{2}(\varphi),\Delta_{2}(\varphi,\varphi^{*}))$. From $\eqref{assf}_{7}$, Lemma \ref{L1} and \eqref{omt1} we estimate
\begin{flalign}\label{10}
\snr{\mathrm{(I)}}\le &c\mint_{B_{r}}\tilde{\omega}(\snr{u-(u)_{B_{r}}})\varphi(\snr{Du}) \ dx \nonumber \\
\le&c\left(\mint_{B_{r}}\tilde{\omega}(\snr{u-(u)_{B_{r}}}) \ dx\right)^{\frac{\delta_{g}}{1+\delta_{g}}}\left(\mint_{B_{r}}\varphi(\snr{Du})^{1+\delta_{g}} \ dx\right)^{\frac{1}{1+\delta_{g}}}\nonumber\\
\le& c\varepsilon^{\frac{\alpha\delta_{g}}{1+\delta_{g}}}\mint_{B_{2r}}\varphi(\snr{Du}) \ dx,
\end{flalign}
with $c=c(n,N,\nu,L,\Delta_{2}(\varphi),\Delta_{2}(\varphi,\varphi^{*}))$. By $\eqref{assf}_{6}$ we now have
\begin{flalign}\label{12}
\snr{\mathrm{(II)}}\le L\mint_{B_{r}}\tilde{\omega}(\snr{x-x_{0}})\varphi(\snr{Du}) \ dx\le cr^{\alpha}\mint_{B_{2r}}\varphi(\snr{Du}) \ dx,
\end{flalign}
with $c=c(n,N,L)$. The $\omega$-minimality of $u$ and \eqref{11} render
\begin{flalign}\label{13}
\snr{\mathrm{(III)}}\le& \omega(r)\mint_{B_{r}}f(x,v,Dv) \ dx \nonumber \\
\le &L\omega(r)\mint_{B_{r}}\varphi(\snr{Dv}) \ dx \le \frac{L^{2}}{\nu}r^{\gamma}\mint_{B_{r}}\varphi(\snr{Du}) \ dx.
\end{flalign}
Term $\mathrm{(IV)}$ can be estimated as term $\mathrm{(I)}$, but this time we need \eqref{omt2}, \eqref{11} and Lemma \ref{L2}:
\begin{flalign}\label{14}
\snr{\mathrm{(IV)}}\le &c \left(\mint_{B_{r}}\tilde{\omega}(\snr{v-(v)_{B_{r}}}) \ dx\right)^{\frac{\sigma_{g}}{1+\sigma_{g}}}\left(\mint_{B_{r}}\varphi(\snr{Dv})^{1+\sigma_{g}} \ dx\right)^{\frac{1}{1+\sigma_{g}}}\nonumber \\
\le &cr^{\frac{\alpha\sigma_{g}}{1+\sigma_{g}}}\left(\mint_{B_{r}}\varphi(\snr{Du})^{1+\sigma_{g}} \ dx\right)^{\frac{1}{1+\sigma_{g}}} \le c \varepsilon^{\frac{\alpha\sigma_{g}}{1+\sigma_{g}}}\mint_{B_{2r}}\varphi(\snr{Du}) \ dx,
\end{flalign}
with $c=c(n,N,\nu,L,\Delta_{2}(\varphi),\Delta_{2}(\varphi,\varphi^{*}))$. By $\eqref{assf}_{6}$ and \eqref{11} we get
\begin{flalign}\label{15}
\snr{\mathrm{(V)}}\le c\mint_{B_{r}}\tilde{\omega}(\snr{x-x_{0}})\varphi(\snr{Dv}) \ dx \le cr^{\alpha}\mint_{B_{2r}}\varphi(\snr{Du}) \ dx,
\end{flalign}
for $c=c(n,N,\nu,L)$. Finally, from $\eqref{assf}_{7}$, \eqref{omt3}, \eqref{11} and Lemma \ref{L2} we estimate
\begin{flalign}\label{16}
\snr{\mathrm{(VI)}}\le c\left(\mint_{B_{r}}\tilde{\omega}(\snr{u-v}) \ dx \right)^{\frac{\sigma_{g}}{1+\sigma_{g}}}\left(\mint_{B_{r}}\varphi(\snr{Dv})^{1+\sigma_{g}} \ dx\right)^{\frac{1}{1+\sigma_{g}}}\le cr^{\frac{\alpha \sigma_{g}}{1+\sigma_{g}}}\mint_{B_{2r}}\varphi(\snr{Du}) \ dx,
\end{flalign}
where $c=c(n,N,\nu,L,\Delta_{2}(\varphi),\Delta_{2}(\varphi,\varphi^{*}))$. Collecting estimates \eqref{10}-\eqref{16} and recalling that, by Lemma \ref{L2}, $\sigma_{g}<\delta_{g}$, we can conclude that
\begin{flalign}\label{17}
\mint_{B_{r}}\snr{V_{\varphi}(Du)-V_{\varphi}(Dv)}^{2} \ dx \le c(r^{\mu}+\varepsilon^{\frac{\alpha\sigma_{g}}{1+\sigma_{g}}})\mint_{B_{2r}}\varphi(\snr{Du}) \ dx,
\end{flalign}
with $\mu:=\min\left\{\frac{\alpha \sigma_{g}}{1+\sigma_{g}},\gamma\right\}$ and $c=c(n,N,\nu,L,\Delta_{2}(\varphi),\Delta_{2}(\varphi,\varphi^{*}),c_{\varphi})$. Notice that, since $f(\cdot)$ satisfies \eqref{assf2} and $v$ solves \eqref{fz}, we can apply \eqref{aff6}, thus obtaining, for $0<t<s<r$,
\begin{flalign}\label{19}
\int_{B_{t}}\varphi(\snr{Dv}) \ dx \le c(t/s)^{n}\int_{B_{s}}\varphi(\snr{Dv}) \ dx,
\end{flalign}
with $c=c(n,N,\nu,L,\Delta_{2}(\varphi),\Delta_{2}(\varphi,\varphi^{*}),c_{\varphi})$. Now we fix $\tau \in \left(0,\frac{1}{4}\right)$ and use \eqref{phi9} and \eqref{17} to estimate
\begin{flalign}\label{18}
\int_{B_{2\tau r}}\varphi(\snr{Du}) \ dx \le &c \left\{\int_{B_{\tau r}}\snr{V_{\varphi}(Du)-V_{\varphi}(Dv)}^{2} \ dx +\int_{B_{2\tau r}}\varphi(\snr{Dv}) \ dx\right\}\nonumber \\
\le &c \left\{\int_{B_{r}}\snr{V_{\varphi}(Du)-V_{\varphi}(Dv)}^{2} \ dx +\tau^{n}\int_{B_{r}}\varphi(\snr{Dv}) \ dx\right\}\nonumber \\
\le &c\left(r^{\mu}+\varepsilon^{\frac{\alpha\sigma_{g}}{1+\sigma_{g}}}+\tau^{n}\right)\int_{B_{2r}}\varphi(\snr{Du}) \ dx\nonumber \\
\le& \tau^{n-\sigma}\left(cr^{\mu}\tau^{\sigma-n}+c\varepsilon^{\frac{\alpha\sigma_{g}}{1+\sigma_{g}}}\tau^{\sigma-n}+c\tau^{\sigma}\right)\int_{B_{2r}}\varphi(\snr{Du})\ dx,
\end{flalign}
for $c=c(n,N,\nu,L,\Delta_{2}(\varphi),\Delta_{2}(\varphi,\varphi^{*}),c_{\varphi})$ and any $\sigma \in (0,n)$. For the ease of notation, set $2r=\varrho$. In these terms, \eqref{18} reads as
\begin{flalign}\label{20}
\int_{B_{\tau \varrho}}\varphi(\snr{Du}) \ dx\le \tau^{n-\sigma}\left(c\varrho^{\mu}\tau^{\sigma-n}+c\varepsilon^{\frac{\alpha\sigma_{g}}{1+\sigma_{g}}}\tau^{\sigma-n}+c\tau^{\sigma}\right)\mint_{B_{\varrho}}\varphi(\snr{Du}) \ dx.
\end{flalign}
Now fix $\tau=\tau(n,N,\nu,L,\Delta_{2}(\varphi),\Delta_{2}(\varphi,\varphi^{*}),c_{\varphi},\sigma) \in \left(0,\frac{1}{4}\right)$ such that $c\tau^{\sigma}<\frac{1}{3}$, $\varepsilon \in (0,1)$ so small that $c\varepsilon^{\frac{\alpha\sigma_{g}}{1+\sigma_{g}}} \tau^{\sigma-n}<1/3 $ and a threshold radius $0<2r<R_{*}\le 1$ such that $R_{*}^{\mu}c\tau^{\sigma-n}<1/3$. Here we see that $\varepsilon=\varepsilon(n,N,\nu,L,\Delta_{2}(\varphi),\Delta_{2}(\varphi,\varphi^{*}),c_{\varphi},\sigma)$,\\ $R_{*}=R_{*}(n,N,\nu,L,\Delta_{2}(\varphi),\Delta_{2}(\varphi,\varphi^{*}),c_{\varphi},\sigma)$ and $\sigma\in (0,n)$ is still to be fixed. With these specifics, \eqref{20} becomes
\begin{flalign}\label{21}
\int_{B_{\tau \varrho}}\varphi(\snr{Du}) \ dx \le \tau^{n-\sigma}\int_{B_{\varrho}}\varphi(\snr{Du}) \ dx.
\end{flalign}
Averaging in \eqref{21}, taking $\sigma \in (0,1)$ and recalling the notation adopted in \eqref{small}, we obtain
\begin{flalign*}
E(B_{\tau\varrho})\le \tau^{-\sigma}E(B_{\varrho})=\tau^{1-\sigma}\tau^{-1}E(B_{\varrho})<\tau^{1-\sigma}\tau^{-1}\varphi(\varepsilon/\varrho)\le \tau^{1-\sigma}\varphi\left(\frac{\varepsilon}{\tau \varrho}\right)\le \varphi\left(\frac{\varepsilon}{\tau \varrho}\right),
\end{flalign*}
thus $E(B_{\tau \varrho})<\varphi\left(\frac{\varepsilon}{\tau \varrho}\right)$, so iterations are legal. In particular, for $\kappa \in \N$ we obtain
\begin{flalign}\label{22}
\int_{B_{\tau^{\kappa}\varrho}}\varphi(\snr{Du}) \ dx \le \tau^{\kappa(n-\sigma) }\int_{B_{\varrho}}\varphi(\snr{Du}) \ dx.
\end{flalign}
If $0<s< \varrho\le R_{*}$, we can easily find a $\kappa \in \N$ such that $\tau^{\kappa+1}\varrho\le s\le\tau^{\kappa}\varrho$, so, using \eqref{22} we obtain
\begin{flalign}\label{23}
\int_{B_{s}}\varphi(\snr{Du}) \ dx \le& \int_{B_{\tau^{\kappa}\varrho}}\varphi(\snr{Du}) \ dx \le \tau^{\sigma-n}\tau^{(\kappa +1)(n-\sigma)}\int_{B_{\varrho}}\varphi(\snr{Du}) \ dx\nonumber \\
\le& c(s/\varrho)^{n-\sigma}\int_{B_{\varrho}}\varphi(\snr{Du}) \ dx,
\end{flalign}
for $c=c(n,N,\nu,L,\Delta_{2}(\varphi),\Delta_{2}(\varphi,\varphi^{*}),\sigma)$. Now, the continuity of Lebesgue's integral renders that if $E(B_{\varrho}(x_{0}))<\varphi(\varepsilon/\varrho)$, then $E(B_{\varrho}(y))<\varphi(\varepsilon/\varrho)$ for all $y$ in a neighborhood $I$ of $x_{0}$, see \cite[Chapter 9]{G}, so we can conclude that, for those $y$ there holds
\begin{flalign}\label{24}
\int_{B_{s}(y)}\varphi(\snr{Du}) \ dx \le c(s/\varrho)^{n-\sigma}\int_{B_{\varrho}(y)}\varphi(\snr{Du}) \ dx,
\end{flalign}
for all $0<s\le \varrho\le R_{*}\le 1$, provided \eqref{small} holds. Let us get rid of the restriction $\varrho\le R_{*}$. We distinguish two scenarios: $0<s\le R_{*}<\varrho\le 1$ and $0<R_{*}<s< \varrho\le 1$. In the first case, by \eqref{24} we have that
\begin{flalign*}
\int_{B_{s}(y)}\varphi(\snr{Du}) \ dx \le& c(s/R_{*})^{n-\sigma}\int_{B_{R_{*}}(y)}\varphi(\snr{Du}) \ dx\\
\le& c (s/\varrho)^{n-\sigma}(\varrho/R_{*})^{n-\sigma}\int_{B_{\varrho}(y)}\varphi(\snr{Du}) \ dx\le c(s/\varrho)^{n-\sigma}\int_{B_{\varrho}(y)}\varphi(\snr{Du}) \ dx,
\end{flalign*}
for $c=c(n,N,\nu,L,\Delta_{2}(\varphi),\Delta_{2}(\varphi,\varphi^{*}),\sigma)$, while in the second,
\begin{flalign*}
\int_{B_{s}}\varphi(\snr{Du}) \ dx\le& (s/\varrho)^{n-\sigma}(R_{*}/\varrho)^{n-\sigma}(s/R_{*})^{n-\sigma}\int_{B_{\varrho}}\varphi(\snr{Du}) \ dx\\
\le& R_{*}^{\sigma-n}(s/\varrho)^{n-\sigma}\int_{B_{\varrho}}\varphi(\snr{Du}) \ dx\le c(s/\varrho)^{n-\sigma}\int_{B_{\varrho}}\varphi(\snr{Du}) \ dx,
\end{flalign*}
for $c=c(n,N,\nu,L,\Delta_{2}(\varphi),\Delta_{2}(\varphi,\varphi^{*}),\sigma)$, hence, if \eqref{small} is in force, the intrinsic Morrey decay for $\varphi(\snr{Du})$ in \eqref{24} holds for any couple $0<s<\varrho\le 1$. \\\\
\emph{Step 2: the singular set}. Define the set
\begin{flalign*}
\Omega_{u}:=\left\{x_{0}\in \Omega\colon \mint_{B_{\varrho}(x_{0})}\varphi(\snr{Du}) \ dx<\varphi(\varepsilon/\varrho) \ \mathrm{for \ some \ }\varrho \in (0,1)\right\},
\end{flalign*}
which is open because of the continuity of Lebesgue's integral. Define $\Sigma_{u}:=\Omega\setminus \Omega_{u}$ and notice that, by very definition,
\begin{flalign*}
\Sigma_{u}\subset \left\{x_{0}\in \Omega\colon \limsup_{\varrho\to 0}\frac{1}{\varphi(\varrho^{-1})}\mint_{B_{\varrho}(x_{0})}\varphi(\snr{Du}) \ dx >0\right\}.
\end{flalign*}
As discussed at the beginning of Section \ref{pre}, we know that, for $t\ge 1$, $\varphi(t)\ge \varphi(1)t^{s_{0}}$, thus, recalling that $\varrho\le 1$, $(\varphi(\varrho^{-1}))^{-1}\le c\varrho^{s_{0}}$, so
\begin{flalign*}
\limsup_{\varrho\to 0}\frac{1}{\varphi(\varrho^{-1})}\mint_{B_{\varrho}}\varphi(\snr{Du}) \ dx >0 \Rightarrow \limsup_{\varrho\to 0}\varrho^{s_{0}-n}\int_{B_{\varrho}(x_{0})}\varphi(\snr{Du}) \ dx>0,
\end{flalign*}
and, by Lemma \ref{L1}, the content of the previous display yields that
\begin{flalign*}
\limsup_{\varrho\to 0}\varrho^{s_{0}(1+\delta_{g})-n}\int_{B_{\varrho}(x_{0})}\varphi(\snr{Du})^{1+\delta_{g}} \ dx >0.
\end{flalign*}
All in all, we got that 
\begin{flalign*}
\Sigma_{u}\subset \left\{x_{0}\in \Omega \colon \limsup_{\varrho\to 0}\varrho^{s_{0}(1+\delta_{g})-n}\int_{B_{\varrho}(x_{0})}\varphi(\snr{Du})^{1+\delta_{g}} \ dx >0\right\}=:\Sigma,
\end{flalign*}
and, by Giusti's lemma \cite[Proposition 2.7]{G}, $\dim_{\mathcal{H}}(\Sigma)\le n-(1+\delta_{g})s_{0}<n-s_{0}$, hence $\dim_{\mathcal{H}}(\Sigma_{u})<n-s_{0}$.\\\\
\emph{Step 3: The case $s_{0}(1+\delta_{g})>n$}. As already anticipated at the beginning of \emph{Step 1}, here $u\in C^{0,\lambda}_{\mathrm{loc}}(\Omega,\RN)$, for $\lambda:=1-\frac{n}{s_{0}(1+\delta_{g})}$, by Morrey's embedding theorem. Let us outline the major changes to \emph{Step 1} as to obtain \eqref{21}, the rest being exactly the same. In fact, given the $\lambda$-H\"older continuity of $u$, we no longer need to impose any smallness condition like \eqref{small}. Fix $B_{2r}=B_{2r}(x_{0})\Subset \Omega$, with $0<r\le 1/2$. Let us define $v \in u+W^{1,\varphi}_{0}(B_{r},\RN)$ as in \eqref{fz}. From Lemma \ref{L4}, we see that
\begin{flalign}\label{29}
&\mint_{B_{r}}\varphi(\snr{Du}) \ dx \le c\mint_{B_{2r}}\varphi\left(\left | \frac{u-(u)_{B_{2r}}}{r}\right |\right) \ dx\le c\varphi(r^{\lambda-1}),
\end{flalign}
for $c=c(n,N,\nu,L,\Delta_{2}(\varphi), [u]_{0,\lambda; \tilde{\Omega}})$. Keeping \eqref{29} in mind, quantities \eqref{omt1}-\eqref{omt3} can now be estimated as
\begin{flalign}\label{omt11}
\mint_{B_{r}}\tilde{\omega}(\snr{u-(u)_{B_{r}}}) \ dx \le cr^{\alpha\lambda},
\end{flalign}
with $c=c(L,[u]_{0,\lambda;\tilde{\Omega}})$,
\begin{flalign}\label{omt22}
\mint_{B_{r}}\tilde{\omega}(\snr{v-(v)_{B_{r}}}) \ dx \le c\tilde{\omega}\left(r\varphi^{-1}\left(\mint_{B_{r}}\varphi(\snr{Du}) \ dx\right)\right)\le c\tilde{\omega}\left(r\varphi^{-1}(\varphi(r^{\lambda-1}))\right)\le cr^{\alpha\lambda},
\end{flalign}
where $c=c(n,N,\nu,L,\Delta_{2}(\varphi),\Delta_{2}(\varphi,\varphi^{*}),[u]_{0,\lambda;\tilde{\Omega}})$, and
\begin{flalign}\label{omt33}
\mint_{B_{r}}\tilde{\omega}(\snr{u-v}) \ dx \le& c\tilde{\omega}\left(r\varphi^{-1}\left(\mint_{B_{r}}\varphi(\snr{Du-Dv}) \ dx\right)\right)\nonumber \\
\le& c\tilde{\omega}\left(r\varphi^{-1}\left(\mint_{B_{r}}\varphi(\snr{Du}) \ dx\right)\right)\le cr^{\alpha\lambda},
\end{flalign}
for $c=c(n,N,\Delta_{2}(\varphi),\Delta_{2}(\varphi,\varphi^{*}),[u]_{0,\lambda;\tilde{\Omega}})$. As we did in $\emph{Step 1}$, but this time using \eqref{omt11}-\eqref{omt33}, we estimate
\begin{flalign*}
\mint_{B_{r}}&\snr{V_{\varphi}(Du)-V_{\varphi}(Dv)}^{2} \ dx \le c\mint_{B_{r}}f(x_{0},(u)_{B_{r}},\snr{Du})-f(x_{0},(u)_{B_{r}},\snr{Dv}) \ dx \\
=&c\mint_{B_{r}}f(x_{0},(u)_{B_{r}},\snr{Du})-f(x_{0},u,\snr{Du}) \ dx +c\mint_{B_{r}}f(x_{0},u,\snr{Du})-f(x,u,\snr{Du}) \ dx \\
&+c\mint_{B_{r}}f(x,u,\snr{Du})-f(x,v,\snr{Dv}) \ dx+c\mint_{B_{r}}f(x,v,\snr{Dv})-f(x,(v)_{B_{r}},\snr{Dv}) \ dx \\
&+c\mint_{B_{r}}f(x,(v)_{B_{r}},\snr{Dv})-f(x_{0},(v)_{B_{r}},\snr{Dv}) \ dx\nonumber \\
&+c\mint_{B_{r}}f(x_{0},(v)_{B_{r}},\snr{Dv})-f(x_{0},(u)_{B_{r}},\snr{Dv}) \ dx=\mathrm{(I)}+\mathrm{(II)}+\mathrm{(III)}+\mathrm{(IV)}+\mathrm{(V)}+\mathrm{(VI)},
\end{flalign*}
with $c=c(n,N,\nu,c_{\varphi})$. Proceeding as before, we easily see that
\begin{flalign*}
&\snr{\mathrm{(I)}}\le cr^{\alpha\lambda}\mint_{B_{2r}}\varphi(\snr{Du}) \ dx,\quad  \snr{\mathrm{(II)}}\le cr^{\alpha}\mint_{B_{2r}}\varphi(\snr{Du}) \ dx,\\ &\mathrm{(III)}\le cr^{\gamma}\mint_{B_{2r}}\varphi(\snr{Du}) \ dx,\quad \snr{\mathrm{(IV)}}\le cr^{\frac{\alpha\lambda \sigma_{g}}{1+\sigma_{g}}}\mint_{B_{2r}}\varphi(\snr{Du}) \ dx,\\ 
&\snr{\mathrm{(V)}}\le cr^{\alpha}\mint_{B_{2r}}\varphi(\snr{Du}) \ dx, \quad  \snr{\mathrm{(VI)}}\le cr^{\frac{\alpha\lambda\sigma_{g}}{1+\sigma_{g}}}\mint_{B_{2r}}\varphi(\snr{Du}) \ dx,
\end{flalign*}
for $c=c(n,N,\nu,L,\Delta_{2}(\varphi),\Delta_{2}(\varphi,\varphi^{*}),c_{\varphi}, [u]_{0,\lambda;\tilde{\Omega}})$. Defining $\mu:=\frac{\alpha\lambda \sigma_{g}}{1+\sigma_{g}}$ and merging the content of the above two displays we can conclude that
\begin{flalign}\label{36}
\mint_{B_{r}}\snr{V_{\varphi}(Du)-V_{\varphi}(Dv)}^{2} \ dx \le cr^{\mu}\mint_{B_{2r}}\varphi(\snr{Du}) \ dx.
\end{flalign}
Fix $\tau \in \left(0,\frac{1}{4}\right)$ and recall \eqref{19} and \eqref{36} to obtain
\begin{flalign*}
\int_{B_{2\tau r}}\varphi(\snr{Du}) \ dx \le &c \left\{\int_{B_{r}}\snr{V_{\varphi}(Du)-V_{\varphi}(Dv)}^{2} \ dx +\int_{B_{2\tau r}}\varphi(\snr{Dv}) \ dx \right\}\nonumber \\
\le &\tau^{n-\sigma}\left(cr^{\mu}\tau^{\sigma-n}+c\tau^{\sigma}\right)\mint_{B_{2r}}\varphi(\snr{Du}) \ dx,
\end{flalign*}
for $c=c(n,N,\nu,L,\Delta_{2}(\varphi),\Delta_{2}(\varphi,\varphi^{*}),c_{\varphi},[u]_{0,\lambda;\tilde{\Omega}})$. Setting again $2r=\varrho$, $\tau \in \left(0,\frac{1}{4}\right)$ small enough so that $c\tau^{\sigma}<1/2$ and a threshold radius $0<r\le R_{*}\le \frac{1}{2}$ so that $cR_{*}^{\mu}\tau^{\sigma-n}<1/2$, we obtain
\begin{flalign*}
\int_{B_{\tau \varrho}}\varphi(\snr{Du}) \ dx\le \tau^{n-\sigma}\int_{B_{\varrho}}\varphi(\snr{Du}) \ dx,
\end{flalign*}
which is \eqref{21} of \emph{Step 1}. The rest is actually the same.
\\\\
\emph{Step 4: Partial H\"older continuity of $V_{\varphi}(Du)$}. Let us fix an open subset $\tilde{\Omega}\Subset \Omega_{u}$ and $B_{\varrho}\Subset \tilde{\Omega}$. As a consequence of \eqref{24}, after a standard covering argument we have
\begin{flalign}\label{25}
\mint_{B_{\varrho}}\varphi(\snr{Du}) \ dx \le c\varrho^{-\kappa} \quad \mathrm{for \ all \ }\kappa \in \N
\end{flalign}
with $c=c(n,N,\nu,L,\Delta_{2}(\varphi),\Delta_{2}(\varphi,\varphi^{*}),c_{\varphi},\dist(\tilde{\Omega},\partial \Omega_{u}),\kappa)$. From \eqref{incdec1}, the standard Poincar\'e's inequality and \eqref{25}, we see that
\begin{flalign*}
\varrho^{\kappa-s_{0}-n}\int_{B_{\varrho}}\snr{u-(u)_{B_{\varrho}}}^{s_{0}} \ dx\le c\varrho^{\kappa}\mint_{B_{\varrho}}\snr{Du}^{s_{0}} \ dx \le c\varrho^{\kappa}\mint_{B_{\varrho}}[\varphi(\snr{Du})+1 ]\ dx\le c,
\end{flalign*}
thus, after covering, for any fixed $\kappa \in (0,s_{0})$, by the integral characterization of H\"older continuity due to Campanato and Meyers we have $u \in C^{0,\frac{s_{0}-\kappa}{s_{0}}}_{\mathrm{loc}}(\Omega_{u},\RN)$, hence $u$ is locally H\"older continuous at any positive exponent (less than $1$) over $\Omega_{u}$. In particular, \eqref{25} allows selecting $\kappa\in \N$ such that $\varsigma:=\mu-\kappa>0$ in \eqref{36}. Finally, we recall that, if $v$ solves \eqref{fz}, then \eqref{aff7} yields, for $0<t<s\le \varrho/2$,
\begin{flalign*}
\mint_{B_{t}}\snr{V_{\varphi}(Dv)-(V_{\varphi}(Dv))_{B_{t}}}^{2} \ dx \le c(t/s)^{\tilde{\nu}}\mint_{B_{s}}\varphi(\snr{Dv}) \ dx,
\end{flalign*}
with $c=c(n,N,\nu,L,\Delta_{2}(\varphi),\Delta_{2}(\varphi,\varphi^{*}),c_{\varphi})$ and $\tilde{\nu}=\tilde{\nu}(n,N,\nu,L,\Delta_{2}(\varphi),\Delta_{2}(\varphi,\varphi^{*}),c_{\varphi})$. Now, collecting all the above informations, for $0<s<\varrho/2$ we have
\begin{flalign}\label{27}
\mint_{B_{s}}&\snr{V_{\varphi}(Du)-(V_{\varphi}(Du))_{B_{s}}}^{2}\ dx \le c\left\{(\varrho/s)^{n}\mint_{B_{\varrho}}\snr{V_{\varphi}(Du)-V_{\varphi}(Dv)}^{2} \ dx\right.\nonumber \\
&\left.+\mint_{B_{s}}\snr{V_{\varphi}(Dv)-(V_{\varphi}(Dv))_{B_{s}}}^{2} \ dx\right\}\le c\left((\varrho/s)^{n}\varrho^{\varsigma}+(s/\varrho)^{\tilde{\nu}}\varrho^{-\kappa}\right),
\end{flalign}
for $c=c(n,N,\nu,L,\Delta_{2}(\varphi),\Delta_{2}(\varphi,\varphi^{*}),c_{\varphi},\kappa,\dist(\tilde{\Omega},\partial \Omega_{u}))$. Notice that there is no loss of generality in assuming that $\tilde{\nu}<1$, while the exponent on the right-hand side of \eqref{27}, $\kappa$, still needs to be fixed. We equalize in \eqref{27} by setting $s:=(\varrho/4)^{1+a}$. Selecting $\kappa=\frac{\varsigma \tilde{\nu}}{2n}$, after standard manipulations we end up with
\begin{flalign}\label{28}
\mint_{B_{s}}\snr{V_{\varphi}(Du)-(V_{\varphi}(Du))_{B_{s}}}^{2} \ dx \le cs^{\frac{\varsigma\tilde{\nu}}{8n}}, 
\end{flalign}
for $c=c(n,N,\nu,L,\Delta_{2}(\varphi),\Delta_{2}(\varphi,\varphi^{*}),c_{\varphi},\dist(\tilde{\Omega},\partial \Omega_{u}))$ and this holds for all $B_{4s}\Subset \tilde{\Omega}$. Now we use again the integral characterization of H\"older continuity to conclude that $V_{\varphi}(Du) \in C^{0,\beta_{0}}_{\mathrm{loc}}(\Omega_{u},\mathbb{R}^{N\times n})$, with $\beta_{0}=\frac{\varsigma\tilde{\nu}}{16n}$. Looking at the dependencies of the quantities involved, we can conclude that $\beta_{0}=\beta_{0}(n,N,\nu,L,\Delta_{2}(\varphi),\Delta_{2}(\varphi,\varphi^{*}),c_{\varphi},\alpha,\gamma)$.\\\\
\emph{Step 5: Partial H\"older continuity of Du}. From \emph{Step 4} we know that $V_{\varphi}(Du)$ is locally H\"older continuous over $\Omega_{u}$. Fix any open subset $\tilde{\Omega}\Subset \Omega_{u}$ and let $x,y\in \tilde{\Omega}$. From \eqref{aff40} and the $\beta_{0}$-H\"older continuity of $V_{\varphi}(Du)$, we immediately have
\begin{flalign*}
\snr{Du(x)-Du(y)}=&\snr{V_{\varphi}^{-1}((V_{\varphi}(Du))(x))-V_{\varphi}^{-1}((V_{\varphi}(Du))(y))}\nonumber \\
\le &[V_{\varphi}^{-1}]_{0,\beta}\snr{(V_{\varphi}(Du))(x)-(V_{\varphi}(Du))(y)}^{\beta}\nonumber \\
\le &[V_{\varphi}^{-1}]_{0,\beta}[V_{\varphi}(Du)]_{0,\beta_{0}}\snr{x-y}^{\beta\beta_{0}}.
\end{flalign*}
Hence, setting $\beta':=\beta\beta_{0}$ we can conclude that $Du \in C^{0,\beta'}_{\mathrm{loc}}(\Omega_{u})$ and, looking at the dependecies of $\beta$ and $\beta_{0}$, we obtain that $\beta'=\beta'(n,N,\nu,L,\Delta_{2}(\varphi),\Delta_{2}(\varphi,\varphi^{*}),c_{\varphi},\alpha,\gamma)$.
\begin{remark}
\emph{In the scalar case $N=1$, Theorem \ref{reg} holds also in case $f$ does not satisfy assumption \eqref{assf2}. In fact the quasilinear structure is required in the vectorial case $N>1$ as to obtain the results in Proposition \ref{aff30}, while, if $N=1$ we can use Lieberman's work \cite{L} to obtain the same reference estimates even if the integrand does not depend on the modulus of the gradient.}
\end{remark}

\end{document}